\documentclass[10pt]{article}
\usepackage{amsmath,amssymb,amsfonts,amsthm}
\usepackage{mathrsfs}
\usepackage{indentfirst}
\usepackage[pdftex]{color,graphicx}
\usepackage[pdftex,bookmarks,unicode,colorlinks]{hyperref}
\usepackage{enumerate} 
\usepackage[numbers]{natbib} 
\usepackage{setspace} 
\usepackage{lineno} 
\usepackage[title]{appendix} 
\usepackage{extarrows}
\usepackage{bm}

\newcommand{\widesim}[2][3.0]{
  \mathrel{\overset{#2}{\scalebox{#1}[1]{$\sim$}}}
}

\usepackage{caption} 
\makeatletter
\let\@fnsymbol\@arabic 
\makeatother

\theoremstyle{plain}
\newtheorem{theorem}{Theorem}[section]
\newtheorem{proposition}[theorem]{Proposition}
\newtheorem{corollary}[theorem]{Corollary}
\newtheorem{lemma}[theorem]{Lemma}

\theoremstyle{remark}
\newtheorem{remark}[theorem]{Remark}
\newtheorem{example}[theorem]{Example}

\theoremstyle{definition}

\newtheorem{assumption}{Assumption}

\newcommand{\E}{\mathbf E}
\renewcommand{\P}{\mathbf P}
\newcommand{\T}{\mathbb T}
\newcommand{\R}{\mathbb R}
\newcommand{\N}{\mathbb N}
\newcommand{\Z}{\mathbb Z}
\newcommand{\Cp}{\mathbb C}
\renewcommand{\S}{\mathbb S}
\newcommand{\Bo}{{B\setminus\{0\}}}
\newcommand{\Ro}[1]{\mathbb R^{#1}\setminus \{0\}}
\renewcommand{\C}{\mathcal C}
\newcommand{\D}{\mathcal D}
\newcommand{\A}{\mathcal A}
\renewcommand{\L}{\mathcal L}
\newcommand{\e}{\epsilon}
\newcommand{\F}{\mathcal F}
\newcommand{\B}{\mathcal B}

\newcommand{\ind}{\mathbf 1}
\newcommand{\DivEps}[1]{\frac{{#1}}{\epsilon}}

\newcommand{\M}{\mathcal M}

\numberwithin{equation}{section} 

\begin{document}

\title{\bf\Large 
Homogenization of Non-symmetric Jump Processes
}
\author{\bf{\normalsize{
Qiao Huang\footnote{School of Mathematics and Statistics \& Center for Mathematical Sciences, Huazhong University of Science and Technology, Wuhan, Hubei 430074, P.R. China. Email: \texttt{hq932309@alumni.hust.edu.cn}} $^,$\footnote{Present Address: Division of Mathematical Sciences, School of Physical and Mathematical Sciences, Nanyang Technological University, 21 Nanyang Link, Singapore 637371. Email: \texttt{qiao.huang@ntu.edu.sg}},
Jinqiao Duan\footnote{Department of Applied Mathematics, Illinois Institute of Technology, Chicago, IL 60616, USA. Email: \texttt{duan@iit.edu}},
Renming Song\footnote{Department of Mathematics, University of Illinois at Urbana-Champaign, Urbana, IL 61801, USA. Email: \texttt{rsong@illinois.edu}}}}}

\date{}
\maketitle
\vspace{-0.3in}
\begin{abstract}
  We study the homogenization for a class of non-symmetric pure jump Feller processes. The jump intensity involves periodic and aperiodic constituents, as well as oscillating and non-oscillating constituents. This means that the noise can come both from the underlying periodic medium and from external environments, and is allowed to have different scales. 
  It turns out that the Feller process converges in distribution, as the scaling parameter goes to zero, to a L\'evy process. 
  As special cases of our result, some homogenization problems studied in previous works can be recovered. We also generalize the approach to the homogenization of symmetric stable-like processes with variable order. Moreover, we present some numerical experiments to demonstrate the usage of our homogenization results in the numerical approximation of first exit times.
\end{abstract}

\section{Introduction}

As a subclass of Markov processes, Feller processes possess lots of nice properties in both probabilistic and analytic aspects \cite{App09,BSW13,EK09,Kal06}. 
The generator of a Feller process is in general a nonlocal operator. It looks locally like the generator of a L\'evy process, in the sense that it is given by a L\'evy-Khintchine type representation with an $x$-dependent L\'evy triplet $(b(x),a(x),\eta(x,\cdot))$.
For this reason, Feller processes are sometimes called L\'evy-type processes or jump-diffusions, their generators are called L\'evy-type operators. Feller processes with no diffusion parts at all, i.e., $a\equiv0$, are called \emph{(pure) jump processes}. If the generator of a Feller process is non-symmetric as an operator, the process is called \emph{non-symmetric}.

Homogenization problems arise from the study of porous media, composite materials and other physical and engineering systems \cite{BLP78,CD99,CP99}. Generally speaking, in a periodic structure, such as medium or material, the heterogeneities are relatively small compared to its global dimension. Thus, two scales characterize the motion of particles in the structure: the microscopic one describing the heterogeneities, and the macroscopic one describing the global behavior of particles. 
The aim of homogenization is precisely to give the macroscopic properties of the particles by taking into account the properties of the microscopic structure.

In this paper, we focus on the homogenization of jump processes, periodic in space and locally periodic in noise.
Consider a pure jump process in a periodic medium. The drift $b(x)$ and the jump kernel $\eta(x,\cdot)$ are periodic and have a small scale in the spatial variable $x$, due to heterogeneities. In mathematical formulation, the small scale is represented by a small parameter $\e>0$. From the realistic point of view, the noise may occur not only from the underlying periodic medium, but also from external environments. So we may assume that the jump kernel $\eta(x,dz)$ have mixed scales in the noise variable $z$.
These suggest that the generators of jump processes in periodic medium will take the following form:
\begin{equation}\label{Ae}
  \begin{split}
     \A^\e f(x) =&\ \int_{\Ro d} \left[ f( x+z)-f(x) - z\cdot\nabla f(x) \ind_{[1,2)}(\alpha)\ind_B(z) \right] \kappa(x/\e,z,z/\e)J(z)dz \\
       & + \left( \textstyle{\frac{1}{\e^{\alpha-1}}} b(x/\e) + c(x/\e) \right) \cdot \nabla f(x) \ind_{(1,2)}(\alpha).
  \end{split}
\end{equation}
Here and after, we denote by $B$ the unit open ball in $\R^d$, and by $S:=\partial B$ the unit sphere. 
Besides, the drift functions $b,c:\R^d\to\R^d$ are Borel measurable and periodic, $J:\Ro d\to(0,\infty)$ is the density of a (not necessarily symmetric) $\alpha$-stable L\'evy measure \cite{Sat99}, with $\alpha\in(0,2)$. More precise assumptions will be made in the next section. 

The \emph{jump coefficient} $\kappa:\R^d\times(\Ro d)\times(\Ro d)\to [0,\infty)$ is a Borel measurable function, periodic in its first variable. The normal scale of $\kappa$ in $z$ corresponds to the noise constituent coming from external environments. Furthermore, we assume that there exists a function $\kappa^*:\R^d\times(\Ro d)\times(\Ro d)\times(\Ro d)\to [0,\infty)$ periodic in its first and third variables, such that $\kappa(x,z,u)=\kappa^*(x,z,u,u)$. In this case, the jump coefficient $\kappa(x/\e,z,z/\e)$ in \eqref{Ae} is locally periodic in noise variable $z$. It means that the small noise scale can be decomposed into two constituents, corresponding to the periodic medium and the external environments respectively. 

Under some regularity assumptions (see the next section), each L\'evy-type operator $\A^\e$ can generate a Feller process on $\R^d$, say $X^\e$. Our aim is to identify the limit of $X^\e$, when the scaling parameter $\e$ goes to zero. It turns out that (see Theorem \ref{homogenization}), the limit process $X$, in the sense of convergence in distribution, is a L\'evy process with the following generator:
\begin{equation*}
  \begin{split}
     \bar\A f(x) =&\ \int_{\Ro d} \left[ f( x+z)-f(x) - z\cdot\nabla f(x) \ind_{[1,2)}(\alpha)\ind_{\{|z|<1\}} \right] \bar\kappa(z) J(z)dz \\
     &\ + \bar c\cdot \nabla f(x) \ind_{(1,2)}(\alpha),
  \end{split}
\end{equation*}
where $\bar\kappa$ is a homogenized jump coefficient related to the function $\kappa$, $\bar c$ is a homogenized constant.

Homogenization of Feller processes with jumps has been investigated by a number of authors. To name a few, the paper \cite{HIT77} considered the one-dimensional pure-jump case, and \cite{Fra06} studied the homogenization of stable-like processes with variable order. See also \cite{Tom92} for a multi-dimensional generalization with diffusion terms involved. The paper \cite{SU21} investigated the homogenization problem of a class of pure-jump L\'evy processes using a pure analytical approach --- Mosco convergence. Recently, the same authors studied in \cite{HDS22} the periodic homogenization of SDEs with jump noise and corresponding nonlocal PDEs. Under the setting of the present paper, many homogenization problems in the literature mentioned above can be recovered. See Section \ref{example} for comparisons.

The sequel of this paper is organized as follows. In the next section, we list our main assumptions, and prove some preliminary results such as the well-posedness of martingale problems, invariance and ergodicity, etc. Some technical results will be left into appendices, with no affect to the smoothness of reading. In Section \ref{main}, we prove our main result to identify the homogenization limit. Some examples of resolving the homogenization problems in previous works are presented in Section \ref{example}. Section \ref{stable-like} is devoted to the stable-like case that cannot be covered by previous sections. Some numerical experiments for visualizing the homogenization result and approximating the first exit time are also provided in this section. 

\section{General assumptions and preliminary results}\label{Kernel&Ergodicity}

In the section, we collect general assumptions and some results we need. The most crucial results are Corollary \ref{ergodic} and \ref{Poisson}. The former allows us to pass the functional convergence in the main theorem in next section, while the latter is the well-posedness of a Poisson equation that will be used to deal with the drift $\frac{1}{\e^{\alpha-1}}b$. Most proofs in this section are quite short. We put other auxiliary but technical results into Appendix.

\subsection{Assumptions}
Firstly, we list our assumptions on the functions $b,c$, $\kappa$ (or $\kappa^*$) and $J$.
\begin{assumption}\label{bc}
  The functions $b,c$ are in the H\"older class $\C^\beta$ for some $\beta\in(0,1)$, and they are periodic of period 1.
\end{assumption}
\begin{assumption}
  The function $(x,z,u,v)\to\kappa^*(x,z,u,v)$ is periodic of period 1 in $x$ and $u$; the function $\kappa(x,z,u):=\kappa^*(x,z,u,u)$ satisfies that there exist constants $\kappa_1,\kappa_2,\kappa_3>0$, such that for the same $\beta$ of \ref{bc}, and all $x,x_1,x_2\in\R^d$ and $z,u\in\Ro d$,
  \begin{gather}
    \kappa_1\le\kappa(x,z,u)\le\kappa_2, \label{non-degenerate} \\
    |\kappa(x_1,z,u)-\kappa(x_2,z,u)|\le \kappa_3|x_1-x_2|^\beta. \label{Holder}
  \end{gather}
  There exists a function $(x,z,u)\to\kappa_0(x,z,u)$, periodic of period 1 in $x$ and $u$ and also satisfying \eqref{non-degenerate} and \eqref{Holder}, such that for all $x\in\R^d$ and $z\in\Ro d$,
  \begin{equation}\label{kappa_0}
    |\kappa^*(x,z,z/\e,z/\e) - \kappa_0(x,z,z/\e)| \to 0, \quad \e\to 0^+.
  \end{equation}
  We assume also that there exists a function $\tilde\kappa:\R^d\times(\Ro d)\to [0,\infty)$ such that for all $z\in\Ro d$,
  \begin{equation}\label{ez}
    \sup_{x\in\R^d} |\kappa(x,\e z,z) - \tilde\kappa(x,z)| \to 0, \quad \e\to 0^+.
  \end{equation}
  In the case $\alpha\in(1,2)$ and $b\ne0$, we assume further that $\alpha+\beta\ne 2$ and for all $z\in\Ro d$,
  \begin{equation}\label{ez-1}
    \frac{1}{\e^{\alpha-1}} \sup_{x\in\R^d} | \kappa(x,\e z,z)- \tilde\kappa(x,z)| \to0, \quad \e\to 0^+.
  \end{equation}
\end{assumption}

\begin{assumption}\label{J}
  Assume that the function $J$ is positive homogeneous of degree $-(d+\alpha)$ for some $\alpha\in(0,2)$, that is,
  \begin{equation}\label{homogeneous}
    J(r z)=r^{-(d+\alpha)}J(z), \quad r>0, z\in\Ro d,
  \end{equation}
  and is bounded between two positive constants on the unit $(d-1)$-sphere $S$, i.e., there exist constants $j_1,j_2>0$, such that for all $\xi\in S$,
  \begin{equation}\label{bound}
    j_1 \le J(\xi) \le j_2.
  \end{equation}
  In the case $\alpha = 1$, we assume additionally that for each $x\in\R^d$ and $r_1,r_2\in(0,\infty)$,
  \begin{equation}\label{alpha=1}
    \int_{S} \xi \kappa(x,r_1\xi,r_2\xi) J(\xi) d\xi = 0.
  \end{equation}
\end{assumption}

We denote by $\C(\T^d)$ the space of continuous functions on the flat torus $\T^d:=\R^d/\Z^d$ endowed with the supremum norm $\|f\|_\infty := \sup_{x\in\T^d}|f(x)|$.
Denote by $\C^k(\T^d)$, with integer $k\ge0$, the space of continuous functions on $\T^d$ possessing derivatives of orders not greater than $k$, endowed with the norm $\|f\|_{\C^k} :=\|f\|_\infty + \sum_{|\alpha|\le k}\sup_{x\in\T^d}|\nabla^\alpha f(x)|$. 
For a non-integer $\gamma>0$, the H\"older space $\C^\gamma(\T^d)$ are defined as the subspaces of $\C^{\lfloor\gamma\rfloor}$ consisting of functions whose $\lfloor\gamma\rfloor$-th order partial derivatives are locally H\"older continuous with exponent $\gamma-\lfloor\gamma\rfloor$. The spaces $\C^\gamma(\T^d)$ is a Banach space endowed with the norm $\|f\|_{\C^\gamma} :=\|f\|_{\C^{\lfloor\gamma\rfloor}} + \sup_{x,y\in\T^d,x\ne y}\frac{|f(x)-f(y)|}{|x-y|^{\gamma-\lfloor\gamma\rfloor}}$.
The space of infinitely differentiable functions on $\T^d$ is denoted by $\C^\infty(\T^d)$.

\begin{remark}
  (i). We shall always identify a periodic function on $\R^d$ of period 1 with its restriction on the compact space $\T^d$. Then Assumption \ref{bc} amounts to saying that $b,c\in\C^\beta(\T^d)$.

  (ii). The H\"older exponent $\beta$ in \ref{bc} and \eqref{Holder} does not need to be the same. But in view of the embedding of H\"older spaces on compact spaces, we can assume them to be the same, without losing any generality. The assumption $\alpha+\beta\ne 2$ is due to \cite{Bas09}, 
  whose results will be used in Corollary \ref{Poisson-zero}.

  (iii). 
  The relation \eqref{ez} or \eqref{ez-1} ensures that the function $\tilde\kappa(x,z)$ is periodic in $x$ and also satisfies \eqref{non-degenerate} and \eqref{Holder} with same constants $\beta,\kappa_1,\kappa_2,\kappa_3$. We also remark that only the assumptions \eqref{kappa_0}--\eqref{ez-1} really contribute to the identification of homogenization limit, see Lemma \ref{invatiant-meausre} and the main result Theorem \ref{homogenization}; while all other assumptions are for constructing Feller processes and estimating heat kernels, see Proposition \ref{heat_kernel}.


  (iv). 
  A typical example for assumptions \eqref{ez} and \eqref{ez-1} to hold is that $\kappa(x,z,u)$ can be written as the quotient of two positive homogeneous functions in $z$. In the case that $\alpha\in(1,2)$ and $b\ne 0$, the convergence \eqref{ez-1} implies \eqref{ez}. In this case, there involves singularity in the drift coefficient $\frac{1}{\e^{\alpha-1}}b$, so that we need more regularities for $\kappa$ to cancel that singularity, as we will see in the proof of Theorem \ref{homogenization}.

  (v). The positive homogeneity assumption on $J$ is equivalent to saying that $J$ is the density of an $\alpha$-stable L\'evy measure (cf. \cite[Theorem 14.3]{Sat99}).
  By \eqref{homogeneous}, $J(z) = J(|z|\cdot\frac{z}{|z|}) = |z|^{-(d+\alpha)}J(\frac{z}{|z|})$. Then assumption \eqref{bound} implies
  \begin{equation}\label{j}
    j_1 |z|^{-(d+\alpha)} \le J(z) \le j_2 |z|^{-(d+\alpha)}, \quad z\in\Ro d,
  \end{equation}
  that is, $J$ is comparable with the density of the rotation invariant $\alpha$-stable L\'evy measure.

  (vi). It is easy to verify that the assumptions \eqref{non-degenerate} and \eqref{Holder} for $\kappa$ (and hence the same for $\tilde\kappa$ as we have seen in the third remark), and \eqref{homogeneous}--\eqref{alpha=1} for $J$, ensure all assumptions in \cite{GS19} for $\alpha\in(0,1)\cup(1,2)$ and in \cite{Szc18} for $\alpha=1$. We will use the results therein in the sequel.
\end{remark}

%
%
%


\subsection{Feller processes}

We need some auxiliary operators and processes. In fact, we will rescale the operator $\A^\e$ and its canonical process in a effective fashion. For this purpose, we define the following nonlocal operators for $f\in\C^\infty(\T^d)$, the space of all smooth functions on the flat torus $\T^d$ (i.e., smooth periodic functions of period 1),
\begin{equation}\label{Ae-tilde}
  \begin{split}
    \tilde\A^\e f(x) =&\ \int_{\Ro d} \left[ f( x+z)-f(x) - z\cdot\nabla f(x) \left(\ind_{\{1\}}(\alpha)\ind_B(z) + \ind_{(1,2)}(\alpha)\ind_B(\e z) \right) \right]\kappa(x,\e z,z)J(z)dz  \\
    &\ +\left(b(x)+\e^{\alpha-1}c(x)\right) \cdot \nabla f(x) \ind_{(1,2)}(\alpha), \quad \e>0,
  \end{split}
\end{equation}
\begin{equation}\label{A-tilde}
  \begin{split}
    \tilde\A f(x) =&\ \int_{\Ro d} \left[ f( x+z)-f(x) - z\cdot\nabla f(x) \left(\ind_{\{1\}}(\alpha)\ind_B(z) + \ind_{(1,2)}(\alpha)\right) \right] \tilde\kappa(x,z)J(z)dz \\
  &\ + b(x)\cdot \nabla f(x) \ind_{(1,2)}(\alpha).
  \end{split}
\end{equation}

For notational simplicity, we shall allow the parameter $\e$ to be zero so that $\tilde\A^0:=\tilde\A$.
The periodicity and continuity of the function $x\to\kappa(x,z,u)$ and \eqref{non-degenerate}, \eqref{alpha=1} and \eqref{j} imply that $\tilde\A^\e$, $\e\ge0$, are all well-defined unbounded operators on $(\C(\T^d),\|\cdot\|_\infty)$, whose domains contain $\C^\infty(\T^d)$. 
Moreover, it is easy to verify by \eqref{homogeneous} and \eqref{alpha=1} that for each $\e>0$, the operator $\tilde\A^\e$ is a rescaling of $\A^\e$ in the sense
\begin{equation}\label{rescaling}
  \tilde \A^\e f(x) = \e^{\alpha}(\A^\e f_\e)(\e x), \quad f\in\C^\infty(\T^d).
\end{equation}
Here and after, we denote 
$f_\e(x):=f(x/\e)$.

Denote by $\D=\D(\R_+;\R^d)$ ($\D_{\mathrm{per}}=\D(\R_+;\T^d)$) the space of all $\R^d$- ($\T^d$-) valued c\`adl\`ag functions on $\R_+:=[0,\infty)$, equipped with the Skorokhod topology. The following proposition tells us that the operators $\A^\e$, $\tilde\A^\e$ and $\tilde\A$ can generate Feller processes, and the state space of the Feller processes associated to $\tilde\A^\e$ or $\tilde\A$ can be taken as $\T^d$, which is a compact space. Meanwhile, the heat kernel estimates for $\tilde\A^\e$ and $\tilde\A$ is crucial in proving the ergodicity of the associated processes. We also find the core for $\tilde\A^\e$ and $\tilde\A$, which we will use to show the convergence of the associated invariant measures in the sequel.
\begin{proposition}\label{heat_kernel}
  Assume that \ref{bc}, \ref{J} and \eqref{non-degenerate}, \eqref{Holder} hold with constants $\alpha\in(0,2)$ and $\beta\in(0,1)$.

  (i). For every $\e>0$ and $x\in\R^d$, the martingale problem for $(\A^\e,\delta_x)$ has a unique solution $\P^\e_x$ on $(\D,\B(\D))$. The coordinate process $X^\e$ is an $\R^d$-valued Feller process starting from $x$.

  (ii). For every $\e\in[0,1]$ and $x\in\T^d$, the martingale problem for $(\tilde\A^\e,\delta_x)$ has a unique solution $\tilde\P^\e_x$ on $(\D_{\mathrm{per}},\B(\D_{\mathrm{per}}))$. The coordinate process $\tilde X^\e$ is a $\T^d$-valued Feller process starting from $x$ with generator the closure of $(\tilde\A^\e,\C^\infty(\T^d))$, and has a transition probability density $\tilde p^\e(t;x,y)$, i.e., $\tilde\P^\e_x(\tilde X^\e_t\in A)=\int_A \tilde p^\e(t;x,y)dy$, $A\in\B(\T^d)$. Moreover, \\
  (ii.1). the transition probability density $\tilde p^\e(t;x,y)$ is jointly continuous on $(0,\infty)\times\T^d\times\T^d$; \\
  (ii.2). for every $T>0$, there exists a constant $0<C<1$independent of $\e\in[0,1]$, such that for all $t\in (0,T]$ and $x,y\in \T^d$,
  \begin{equation}\label{heat-kernel-est}
    \tilde p^\e(t;x,y) \ge C\sum_{l\in\Z^d}\left[t^{-d/\alpha}\wedge \left(t |x-y+l|^{-(d+\alpha)}\right)\right].
  \end{equation}
\end{proposition}


\begin{proof}
  All assertions for the case $\alpha\in(0,1)$ follow from \cite[Theorem 1.1, Theorem 1.3, Theorem 1.4, Remark 1.5]{GS19}, the assertions for the case $\alpha=1$ can be found in \cite[Theorem 2.1, Theorem 2.3, Theorem 2.4]{Szc18}. In particular, for these two cases, the constant $C$ in the estimate \eqref{heat-kernel-est} depends only on $(d,\alpha,\beta,\kappa_1,\kappa_2,\kappa_3,j_1,j_2)$ and hence independent of $\e\ge0$, since $\kappa(x,\e z,z)$ and $\tilde\kappa(x,z)$, the only quantities in  $\{\tilde\A^\e:\e\ge0\}$ that depend on $\e$ when $\alpha\in(0,1]$, satisfies \eqref{non-degenerate} and \eqref{Holder} with uniform constants $\kappa_1,\kappa_2$. For the case $\alpha\in(1,2)$, the properties of $\tilde p^\e$ can be found in \cite[Theorem 1.5]{CZ18}, we also put its proof into Appendix for reader's convenience, see Proposition \ref{q-properties}. In particular, by (iii) of Proposition \ref{q-properties}, the constant $C$ of \eqref{heat-kernel-est} depends on $(d,\alpha,\beta,\kappa_1,\kappa_2,\kappa_3,j_1,j_2)$ and the upper bound of the drift of each $\tilde\A^\e$, $\e\in[0,1]$. The drift of $\tilde\A^0 = \tilde\A$ is $b$, while that of $\tilde\A^\e$ with $\e\in(0,1]$ is
  \begin{equation*}
    b(x)+\e^{\alpha-1}c(x) - \int_{1\le |z|<\frac{1}{\e}} z \kappa(x,\e z,z)J(z)dz
  \end{equation*}
  whose absolute value is bounded by $\|b\|_\infty + \|c\|_\infty + \frac{\kappa_2}{\alpha-1}$ uniformly for $\e\in(0,1]$. Thus, $C$ in \eqref{heat-kernel-est} can be chosen as independent of $\e\in[0,1]$.
  The proofs of the rest parts are tedious, especially that $\C^\infty(\T^d)$ is the core of the generators, and we leave them into Appendix, see Proposition \ref{Feller} and Corollary \ref{SDE-weak}.
\end{proof}

Of course, each of the processes $X^\e$, $\tilde X$ and $\tilde X^\e$ is defined on its own stochastic basis. However, by taking product of the probability spaces, it is always possible to assume that:
\begin{center}
  $X^\e$, $\tilde X$ and $\tilde X^\e$, $\e>0$, are all defined on the same probability space $(\Omega,\F,\P)$.
\end{center}
We also assume for simplicity that 
\begin{center}
  $X^\e_0=\tilde X_0=\tilde X^\e_0=0$.
\end{center}

If we identify a periodic function on $\R^d$ of period $\e$ with its restriction on the $\e$-torus $\T^d_\e:=\e\T^d$, then each $\A^\e$ maps $\C^\infty(\T^d_\e)$ into $\C^\infty(\T^d_\e)$ by virtue of the periodicity of $\kappa$ in $x$. In view of this, the canonical process $X^\e$ can also be treated as a processes taking values on $\T^d_\e$, via the quotient map $\R^d \to \T^d_\e$. We will use this treatment \emph{only} in the rest of this section, the benefit is the following relation, which follows from the well-known fact that Feller semigroups and Feller processes are in one-to-one correspondence if we identify the processes that have same finite-dimensional distributions (see, e.g., \cite{BSW13}).

\begin{lemma}\label{rescale}
  We have the following indentity in law
  $$\{\tilde X_t^\e\}_{t\ge0} \stackrel{\mathtt d}{=} \left\{ \textstyle{\DivEps 1} X^\e_{\e^{\alpha} t} \right\}_{t\ge0}, \quad\text{for each } \e>0.$$
\end{lemma}

\begin{proof}
  We derive the generator for the Feller process $\left\{ \textstyle{\DivEps 1} X^\e_{\e^{\alpha} t} \right\}_{t\ge0}$. For $f\in \C^\infty(\T^d)$, by \eqref{rescaling},
  \begin{equation*}
    \lim_{t\downarrow0}\frac{\E^\e_{\e x} \left[ f\left( \DivEps 1 X^\e_{\e^{\alpha} t} \right) \right] -f(x)}{t} = \e^{\alpha} \lim_{t\downarrow0}\frac{\E^\e_{\e x} \left[ f_\e\left( X^\e_{\e^{\alpha} t} \right) \right] -f(x)}{\e^{\alpha}t} = \e^{\alpha}(\A^\e f_\e)(\e x) = \tilde \A^\e f(x).
  \end{equation*}
  Therefore, the Feller semigroup associated to $\left\{ \textstyle{\DivEps 1} X^\e_{\e^{\alpha} t} \right\}_{t\ge0}$ is also generated by the closure of $(\tilde\A^\e,\C^\infty(\T^d))$.
\end{proof}

Denote by $\{\tilde P^\e_t\}_{t\ge0}$ (or $\{\tilde P_t\}_{t\ge0}$) the Feller semigroup of $\tilde X^\e$ (or $\tilde X$). Let $\tilde X^{0}=\tilde X$ and $\tilde P^0_t=\tilde P_t$. Now utilizing the lower bound of the transition probability density $\tilde p^\e(t;x,y)$, we obtain the following exponential ergodicity.
\begin{proposition}\label{exponential-ergodicity}
  For each $\e\in[0,1]$, the Feller process $\tilde X^\e$ possesses a unique invariant probability distribution $\mu_\e$ on $\T^d$. Moreover, there exist positive constants $C$ and $\rho$ which are independent of $\e\in[0,1]$, such that for each periodic bounded Borel function $f$ on $\R^d$ (i.e., $f$ is bounded Borel on $\T^d$),
  \begin{equation*}
    \sup_{x\in\T^d} \left| \tilde P^\e_t f(x)-\int_{\T^d}f(y)\mu_\e(dy) \right| \le C\|f\|_\infty e^{-\rho t}
  \end{equation*}
  for every $t\ge0$.
\end{proposition}
\begin{proof}
  The proof is similar to \cite[Theorem 3.3.2]{BLP78}, see also \cite[Lemma 4.6]{HDS22}. The only problem is to show that the two constants $C$ and $\rho$ can be chosen to be independent of $\e\in[0,1]$. Thanks to the Doeblin-type result in \cite[Theorem 3.3.1]{BLP78}, it suffices to show that the map $\T^d\times\T^d\ni(x,y) \mapsto\tilde p^\e(1;x,y)$ is bounded from below by a positive constant independent of $x, y$ and $\e$. This follows immediately from the transition density estimate in \eqref{heat-kernel-est} together with the compactness of the state space $\T^d$ and the joint continuity of $\tilde p^\e$.
\end{proof}

Denote by $\mu=\mu_0$ the invariant probability measure of $\tilde X$. 

\begin{lemma}\label{invatiant-meausre}
  As $\epsilon\to0^+$, we have the weak convergence $\mu_\epsilon \Rightarrow \mu$.
\end{lemma}

\begin{proof}
  By the argument in \cite[Lemma 2.4]{HP08}, we only need to show that $\tilde P_t^\e f\to \tilde P_tf$ in $\C(\T^d)$ as $\e\to 0^+$ for each $f\in\C(\T^d)$ and $t\ge0$.

  By Proposition \ref{heat_kernel}, we know that $\C^\infty(\T^d)$ is a core for each $\tilde\A^\e$, $\e\ge0$. Now fix an arbitrary $f\in\C^\infty(\T^d)$. If $\alpha\in(0,1]$, then $\|\tilde\A^\e f-\tilde\A f\|_\infty$ as $\e\to0^+$ by dominated convergence and \eqref{ez}. For the case $\alpha\in(1,2)$, we use the following fact from Taylor expansion
  \begin{equation*}
    |f( x+z)-f(x) - z\cdot\nabla f(x)| \le \frac{1}{2} \|f\|_{\C^2} |z|^2 \ind_{\{|z|\le1\}} + 2\|f\|_{\C^1} |z| \ind_{\{|z|>1\}} \le 2\|f\|_{\C^2} (|z|^2 \wedge |z|)
  \end{equation*}
  to derive
  \begin{equation}\label{scaling-diff}
    \begin{split}
       \| \tilde\A^\e f-\tilde\A f \|_\infty \le&\ \| \e^{\alpha-1}c \cdot \nabla f \|_\infty + \sup_{x\in\T^d} \left| \int_{\Ro d} z\cdot\nabla f(x) \left(1-\ind_{B}(\e z)\right) \kappa(x,\e z,z)J(z) dz \right| \\
         & + \sup_{x\in\T^d} \left|\int_{\Ro d} \left[ f( x+z)-f(x) - z\cdot\nabla f(x) \right] (\kappa(x,\e z,z)-\tilde\kappa(x,z))J(z)dz\right| \\
         \le&\ \e^{\alpha-1}\|c\|_\infty \|f\|_{\C^1} + \kappa_2 j_2\|f\|_{\C^1} \int_{|z|\ge 1/\e} \frac{dz}{|z|^{d+\alpha-1}} \\
         & + 2 j_2\|f\|_{\C^2} \int_{\Ro d} \left( \sup_{x\in\T^d} |\kappa(x,\e z,z)-\tilde\kappa(x,z)| \right) (|z|^2\wedge|z|) \frac{dz}{|z|^{d+\alpha}},
    \end{split}
  \end{equation}
  which converges to zero as $\e\to0^+$, by \eqref{ez} and dominated convergence. Now by the Trotter-Kato approximation theorem (see \cite[Theorem III.4.8]{EN00}), $\tilde P_t^\e f\to \tilde P_tf$ in $\C(\T^d)$ as $\e\to 0^+$ for all $f\in\C(\T^d)$, uniformly for $t$ in compact intervals.
\end{proof}

Now using Proposition \ref{exponential-ergodicity} and Lemma \ref{invatiant-meausre}, we can obtain a useful functional convergence theorem.
\begin{corollary}\label{ergodic}
  Let $f$ be a bounded Borel function on $\T^d$.
  Then for every $t>0$,
  \begin{equation*}
    \E \left[ \left| \int_0^t f\left( \frac{X_s^\e}{\epsilon} \right) ds - t\int_{\T^d}f(y)\mu(dy) \right|^2\right] \to 0, \quad\text{as } \epsilon\to 0^+.
  \end{equation*}
  For every $T>0$,
  \begin{equation}\label{est-1}
    \sup_{t\in[0,T]}\left| \int_0^t f\left( \frac{X_s^\e}{\epsilon} \right) ds - t\int_{\T^d}f(y)\mu(dy) \right|\to 0,\quad \text{ in probability } \P,\text{ as } \epsilon\to 0^+.
  \end{equation}
\end{corollary}
\begin{proof}
  We follow the lines of \cite[Proposition 2.4]{Par99}. Fix $\e>0$. Let $\bar f_\e:= f-\int_{\T^d}f(y)\mu_\e(dy)$. 
  By virtue of Lemma \ref{rescale} and Lemma \ref{invatiant-meausre}, to prove the two limits, it suffices to prove that
  \begin{equation}\label{est-6}
    \e^\alpha\int_0^{\e^{-\alpha}t} \bar f_\e(\tilde X^\e_s) ds = \int_0^t \bar f_\e(X^\e_s/\e) ds \to 0,
  \end{equation}
  in $L^2(\Omega,\P)$ and also in probability uniformly in $t\in[0,T]$.
  By Proposition \ref{exponential-ergodicity}, for $0\le s<t$ we have
  \begin{equation}\label{est-4}
    \E \left( \bar f_\e(\tilde X^{\e}_t) \Big| \tilde X^{\e}_s\right) = \int_{\T^d} \bar f_\e(y) \left[\tilde p^\e(t-s,\tilde X_s^{\e},y)dy-\mu_\e(dy)\right] \le C\|\bar f_\e\|_\infty e^{-\rho(t-s)},
  \end{equation}
  and then by the Markov property,
  \begin{equation}\label{est-5}
    \E\left( \bar f_\e(\tilde X^\e_s) \bar f_\e(\tilde X^\e_t) \right) = \E\left[\bar f_\e(\tilde X^\e_s) \E \left(\bar f_\e(\tilde X^{\e}_t)\Big| \tilde X^{\e}_s\right)\right] \le C\|\bar f_\e\|_\infty^2 e^{-\rho(t-s)} \le 4C\|f\|_\infty^2 e^{-\rho(t-s)}.
  \end{equation}
  Hence, if we denote $g_\e(s):=\bar f_\e(X^\e_s/\e)$, then as $\e\to0^+$,
  \begin{equation}\label{est-2}
    \begin{split}
       \E\left[\left| \int_0^t g_\e(s) ds \right|^2\right] &= 2\e^{2\alpha}\int_0^{\e^{-\alpha}t} \int_0^{r} \E \left( \bar f_\e(\tilde X^\e_s) \bar f_\e(\tilde X^\e_r) \right) ds dr \\
         &\le 8C\e^{2\alpha} \|f\|_\infty^2 \int_0^{\e^{-\alpha}t} \int_0^{r} e^{-\rho(r-s)} dsdr \\
         &= 8C\e^{2\alpha} \|f\|_\infty^2 \rho^{-2} \left[ -1+\rho\e^{-\alpha}t + e^{-\rho\e^{-\alpha}t} \right] \\
         &\to 0,
    \end{split}
  \end{equation}
 The first result follows. On the other hand, for any $n\in\N_+$, since $(\lfloor \frac{nt}{T} \rfloor +1) \frac{T}{n}\ge t$,
  \begin{equation*}
    \begin{split}
       \E\left[ \sup_{t\in[0,T]} \left| \int_0^t g_\e(s) ds \right|^2\right] &= \E\left[ \sup_{k=0,\cdots,n} \left| \int_0^{k\frac{T}{n}} g_\e(s) ds \right|^2 + \sup_{t\in[0,T]} \left( \left| \int_0^{\lfloor \frac{nt}{T} \rfloor \frac{T}{n}} g_\e(s) ds \right| + \left| \int_{\lfloor \frac{nt}{T} \rfloor \frac{T}{n}}^t g_\e(s) ds \right| \right)^2 \right] \\
         &\le 3 \E\left[ \sup_{k=0,\cdots,n} \left| \int_0^{k\frac{T}{n}} g_\e(s) ds \right|^2 \right] + 2 \E\left[ \sup_{t\in[0,T]} \left| \int_{\lfloor \frac{nt}{T} \rfloor \frac{T}{n}}^t g_\e(s) ds \right|^2 \right] \\
         &\le 3 \sup_{k=0,\cdots,n} \E\left[ \left| \int_0^{k\frac{T}{n}} g_\e(s) ds \right|^2 \right] + 8 \|f\|_\infty^2 \frac{T^2}{n^2},
    \end{split}
  \end{equation*}
  which goes to zero, by first letting $\e\to0^+$ and applying \eqref{est-2}, and then letting $n\to\infty$. The second result follows by an application of Chebyshev's inequality.
\end{proof}

\begin{remark}\label{remark-2}
  (i). From \eqref{est-6}, we have indeed proved the following ergodicity result: for every $\e\in(0,1]$ and bounded Borel function $f$ on $\T^d$,
  $$\frac{1}{T} \int_0^T f(\tilde X^\e_s) ds \to \int_{\T^d}f(y)\mu_\e(dy), \quad \text{as } T\to\infty, \text{ in } L^2(\Omega,\P).$$
  This result also holds for $\e=0$, since Proposition \ref{exponential-ergodicity} and thereby \eqref{est-4} and \eqref{est-5} are all valid for $\e=0$, we have in a similar way as \eqref{est-2} that, denoting $\bar f:= f-\int_{\T^d}f(y)\mu(dy)$,
  \begin{equation*}
    \E\left[\left| \frac{1}{T}\int_0^T \bar f(\tilde X_s) ds \right|^2\right] = \frac{2}{T^2}\int_0^T \int_0^{r} \E \left( \bar f_\e(\tilde X^\e_s) \bar f_\e(\tilde X^\e_r) \right) ds dr \to 0, \quad \text{as } T\to\infty.
  \end{equation*}

  (ii). In the sequel, we shall use the following variant of \eqref{est-1}: Let $f:\T^d\times \R_+\times\R^d\to\R$ be a bounded Borel function, then for every $T>0$, as $\epsilon\to 0^+$,
  \begin{equation}\label{est-3}
    \sup_{t\in[0,T]}\left| \int_0^t f\left( \frac{X_s^\e}{\epsilon}, \e, z \right) ds - t\int_{\T^d}f(y, \e, z)\mu(dy) \right|\to 0,\quad \text{ in probability } \P.
  \end{equation}
  Clearly, this holds for the case that $f$ is separable as $f(y, \e, z) = f_0(y)g(\e,z)$. The general case follows by first making monotone approximations for positive and negative parts of $f$ by simple functions (linear combinations of indicator functions) respectively, and then applying the monotone convergence theorem.

  (iii). The proof of a similar result in another paper from the same authors, \cite[Proposition 4.8]{HDS22}, is partially incorrect, even though it does not affect the main results therein. The correct proof needs to be conducted in the same way as here.
\end{remark}

\subsection{Nonlocal Poisson equation}
Using the exponential ergodicity, we can also obtain the well-posedness of the nonlocal Poisson equation. Denote by $\C^\gamma_\mu(\T^d)$, $\gamma>0$, the class of all $f\in\C^\gamma(\T^d)$ which are \emph{centered} with respect to the invariant measure $\mu$ in the sense that $\int_{\T^d}f(x)\mu(dx)=0$. It is easy to check that $\C_\mu(\T^d)$ is a closed subset, and hence a sub-Banach space of $\C(\T^d)$ under the norm $\|\cdot\|_\infty$.

\begin{lemma}
  The restrictions $\{\tilde P_t^\mu := \tilde P_t|_{\C_\mu(\T^d)}\}_{t\ge0}$ form a $C_0$-semigroup on the Banach space $(\C_\mu(\T^d),\|\cdot\|_\infty)$, with generator $(\tilde\A_\mu,D(\tilde\A_\mu)):= \overline{(\tilde\A, \C^\infty_\mu(\T^d))}$. Moreover, the set $\{z\in\Cp\mid\mathrm{Re}z>-\rho\}$ is contained in the resolvent set of $\tilde \A_\mu$.
\end{lemma}

\begin{proof}
  Since $\mu$ is invariant with respect to $\{\tilde P_t\}_{t\ge0}$, it is easy to see that $\C_\mu(\T^d)$ is $\{\tilde P_t\}_{t\ge0}$-invariant, in the sense that $\tilde P_t(\C_\mu(\T^d))\subset\C_\mu(\T^d)$ for all $t\ge0$. The first part of the lemma then follows from the corollary in \cite[Subsection II.2.3]{EN00}. By the exponential ergodicity result in Proposition \ref{exponential-ergodicity}, we have
  \begin{equation}\label{P_per}
    \|\tilde P_t^\mu f\|_\infty\le C\|f\|_\infty e^{-\rho t}
  \end{equation}
  for all $f\in\C_\mu(\T^d)$ and $t\ge0$. This yields the second part of the lemma, using \cite[Theorem II.1.10.(ii)]{EN00}.
\end{proof}

\begin{corollary}\label{Poisson}
  Let $\alpha\in(1,2)$. For every $f\in\C^\beta_\mu(\T^d)$, there exists a unique solution in $\C^{\alpha+\beta}_\mu(\T^d)$ to the Poisson equation
  \begin{equation}\label{Poisson-00}
    \tilde\A u+f = 0.
  \end{equation}
\end{corollary}

\begin{proof}
  If $u\in \C^{\alpha+\beta}_\mu(\T^d)$ is a solution, then by \eqref{P_per},
  \begin{equation*}
    \int_0^\infty \tilde P^\mu_t f dt = - \int_0^\infty \tilde P^\mu_t \tilde\A_\mu u\, dt = - \int_0^\infty \frac{d}{dt} \tilde P^\mu_t u\, dt = u - \lim_{t\to\infty}\tilde P^\mu_t u = u.
  \end{equation*}
  This yields the uniqueness. Thanks to Corollary \ref{Poisson-zero}, the existence follows from a standard argument of Fredholm alternative (\cite[Section 5.3]{GT01}).
\end{proof}

According to the terminology of periodic homogenization, we will refer to equation \eqref{Poisson-00} as the \emph{cell problem}.

\section{Homogenization result}\label{main}

In this section we will prove our homogenization result. Before that, some preparations are needed.
%

Firstly, we need a convergence lemma for locally periodic function.

\begin{lemma}\label{convergence-periodic}
  Let $\phi:\R^d\times\R^d\to\R, (x,y)\mapsto \phi(x,y)$ be a function periodic in $y$ with period 1. \\
  (i). Let $1<p<\infty$. Suppose for each $x\in\R^d$, $\phi(x,\cdot) \in L^p([0,1]^d)$, and for each $y\in\R^d$, $\phi(\cdot,y) \in L_{\text{loc}}^{p'}(\R^d)$, where $p'$ is the conjugate of $p$, i.e., $\frac{1}{p}+\frac{1}{p'}=1$. Then for every compact set $K\subset\R^d$, we have
  $$\lim_{\e\to 0^+} \int_K \phi \left(x, \frac{x}{\e} \right) dx = \int_K \int_{\T^d}\phi(x,y)dy dx.$$
  (ii). Suppose for each $x\in\R^d$, $\phi(x,\cdot) \in L^\infty([0,1]^d)$, and for each $y\in\R^d$, $\phi(\cdot,y) \in L^1(\R^d)$. Then we have
  $$\lim_{\e\to 0^+} \int_{\R^d} \phi \left(x, \frac{x}{\e} \right) dx = \int_{\R^d} \int_{\T^d}\phi(x,y)dy dx.$$
\end{lemma}

In the case that the function $\phi$ is separable, that is, $\phi$ is of the form $\phi(x,y) = f(x) g(y)$ with $g$ periodic, the conclusions of the above lemma can be found in \cite[Theorem 2.6]{CD99}. 
The general case can be achieved via standard 
monotone approximations of positive and negative parts of $\phi$ by simple functions and the monotone convergence theorem.

%

Now we are in the position to prove the homogenization result. To get rid of the singularity in the coefficient $\frac{1}{\e^{\alpha-1}}b$ in the case $\alpha\in(1,2)$, we need one more assumption on $b$.

\begin{assumption}\label{center}
  The function $b$ satisfies the \emph{centering condition},
  \begin{equation*}
    \int_{\T^d}b(x)\mu(dx)=0.
  \end{equation*}
\end{assumption}
By virtue of Assumptions \ref{bc}, \ref{center} and Corollary \ref{Poisson}, when $\alpha\in(1,2)$ there exists a function $\hat b\in \C^{\alpha+\beta}_\mu(\T^d)$ that uniquely solves the Poisson equation
\begin{equation}\label{b-hat}
  \tilde\A \hat b+b = 0.
\end{equation}

\begin{theorem}\label{homogenization}
  Suppose that Assumptions \ref{bc}--\ref{center} hold. In the sense of weak convergence on the space $\D(\R_+;\R^d)$, we have
  $$X^\e \ \Rightarrow \bar X, \quad \text{as } \e\to 0^+.$$
  The limit process $\bar X$ is a L\'evy process starting from 0 with L\'evy triplet $(\bar b,0,\bar \nu)$ given by
  \begin{equation}\label{homogenized}\left\{
    \begin{array}{l}
      {\displaystyle \bar b = \ind_{(0,1)}(\alpha) \int_{\Bo} \bar\kappa(z) zJ(z) dz+ \ind_{(1,2)}(\alpha)\bar c,} \\
      {\displaystyle\bar\nu(dz) = \bar\kappa(z) J(z) dz,}
      \end{array} \right.
  \end{equation}
  with homogenized coefficients
  \begin{gather*}
    \bar\kappa(z): = \int_{\T^d} \int_{\T^d} \kappa_0(x,z,u) du \mu(dx), \\
    \bar c := \int_{\T^d}\left(I+\nabla \hat b(x)\right)\cdot c(x)\mu(dx) + \int_{B^c}z \cdot\left( \int_{\T^d} \int_{\T^d} \nabla \hat b(x) \kappa_0(x,z,u) du \mu(dx) \right) J(z) dz,
  \end{gather*}
  where $\mu$ is the invariant probability measure of $\tilde X$ with generator \eqref{A-tilde}, $\hat b$ is uniquely determined by \eqref{b-hat}.
\end{theorem}

\begin{proof}
  (i). We first prove the theorem for the case that $b\equiv0$ or $\alpha\in(0,1]$. By \cite[Theorem 2.44]{BSW13}, we know that the semimartingale characteristics of $X^\e$ relative to the truncation function $\ind_B$ are $(B^\e,0,\nu^\e)$, where
  \begin{equation*}\left\{
    \begin{array}{l}
      {\displaystyle B^\e_t = \ind_{(0,1)}(\alpha) \int_0^t\int_{\Bo}z\kappa^*\left(\DivEps{X^\e_s},z,\DivEps z,\DivEps z\right)J(z)dzds, + \ind_{(1,2)}(\alpha) \int_0^t c\left( \DivEps{X^\e_s} \right)ds } \\
      {\displaystyle \nu^\e(dz,dt) = \kappa^*\left(\DivEps{X^\e_t},z,\DivEps z,\DivEps z\right)J(z)dzdt.}
    \end{array} \right.
  \end{equation*}
  By applying the functional central limit theorem in \cite[Theorem VIII.2.17]{JS87}, we only need to show that for all $t\in\R_+$ and every bounded continuous function $f:\R^d\to\R$ vanishing in a neighbourhood of the origin, the following convergences hold in probability when $\e\to0^+$:
  \begin{align}
    \sup_{0\le s\le t} |B^\e_s-\bar bs| &\to 0, \label{B-conv}\\
    \int_0^t\int_{\Ro d}f(z)\nu^\e(dz,ds) &\to t\int_{\Ro d}f(z)\bar \nu(dz). \label{nu-conv}
  \end{align}
  Clearly, by Corollary \ref{ergodic} we have
  \begin{equation}\label{c-conv}
    \int_0^t c\left( \DivEps{X^\e_s} \right)ds \to t\int_{\T^d}c(x)\mu(dx),\quad \text{in probability}, \text{ as } \epsilon\to 0^+.
  \end{equation}
  When $\alpha\in(0,1)$, we also have the following convergence in probability, uniformly with respect to $t$ in closed intervals,
  \begin{align*}
    &\ \int_0^t\int_{\Bo}z\kappa^*\left( \frac{X^{\e}_s}{\e},z, \frac{z}{\e},\frac{z}{\e} \right)J(z)dzds \\
    \widesim{\e\ll 1} &\ t \int_{\Bo}\left[ \int_{\T^d} \kappa^*\left( x,z, \frac{z}{\e},\frac{z}{\e} \right) \mu(dx) \right] zJ(z) dz \qquad (\text{by \eqref{est-3}}) \\
    \widesim{\e\ll 1} &\ t \int_{\Bo}\left[ \int_{\T^d} \kappa_0\left( x,z, \frac{z}{\e} \right) \mu(dx) \right] zJ(z) dz \qquad (\text{by \eqref{kappa_0} and dominated convergence}) \\
    \xrightarrow{\e\to0^+} &\ t \int_{\Bo}\left[ \int_{\T^d} \int_{\T^d} \kappa_0\left( x,z, u \right) du \mu(dx) \right] zJ(z) dz \qquad (\text{by Lemma \ref{convergence-periodic}.(i)}).
  \end{align*}
  In the second line, to apply \eqref{est-3} we take $f(y,\e,z) = \kappa^*(y,z,z/\e,z/\e)$. In the last line, to apply Lemma \ref{convergence-periodic}.(i) we take $K=B$, and $\phi(z,u)=\kappa_0(x,z,u)zJ(z)$ for fixed $x$. Choose $p'\in(1,\frac{d}{d+\alpha-1})$, then it is easy to verify from \eqref{non-degenerate} and \eqref{j} that for each $u$, $\phi(\cdot,u)\in L^{p'}(K)$, and for each $z$, $\phi(z,\cdot)\in L^p([0,1]^d)$. This proves the assertion \eqref{B-conv}. The assertion \eqref{nu-conv} follows in a similar fashion but with Lemma \ref{convergence-periodic}.(ii) in place of Lemma \ref{convergence-periodic}.(i) and letting $\phi(z,u)=\kappa_0(x,z,u)f(z)J(z)$.

  (ii). We prove the general case that $b\neq0$ and $\alpha\in(1,2)$.
  Define $\hat X_t^{\epsilon}:=X_t^{\epsilon}+\epsilon \hat b_\e \left(X_t^{\epsilon}\right)$, the boundedness of $\hat b$ yields that $\hat X^\e$ and $X^\e$ have the same limit. Applying Corollary \ref{SDE-weak}, Lemma \ref{Ito} and \eqref{rescaling}, we have
  \begin{equation*}
    \begin{split}
      \hat X_t^{\epsilon} = &\ \int_0^t c\left(\frac{X_{s}^{\e}}{\e}\right) ds + \int_0^t \frac{1}{\e^{\alpha-1}}\left( \tilde\A^\e \hat b - \tilde\A \hat b \right) \left(\frac{X_{s}^{\epsilon}}{\epsilon}\right) ds \\
        &\ +\int_0^t \int_0^\infty \int_{\Ro d}\epsilon\left[\hat b_\epsilon\left(X_{s-}^{\epsilon}+ \ind_{[0,\kappa(X_{s-}^\e/\e,z,z/\e))}(r)z \right)-\hat b_\epsilon\left(X_{s-}^{\epsilon}\right) \right] \tilde N(dz,dr,ds) \\
        &\ +\int_0^t \int_0^\infty \int_{\Bo} \ind_{[0,\kappa(X_{s-}^\e/\e,z,z/\e))}(r)z \tilde N(dz,dr,ds) \\
        &\ +\int_0^t \int_0^\infty \int_{B^c} \ind_{[0,\kappa(X_{s-}^\e/\e,z,z/\e))}(r)z N(dz,dr,ds) \\
        =: &\ I_1^\e(t)+I_2^\e(t)+I_3^\e(t)+I_4^\e(t)+I_5^\e(t),
    \end{split}
  \end{equation*}
  where $N$ is a Poisson random measure on $\R^d\times[0,\infty)\times[0,\infty)$ with intensity measure $J(z)dz\times m\times m$ and $\tilde N$ is the associated compensated Poisson random measure. The convergence of $I_1^\e$ is shown in \eqref{c-conv}. For $I_2^\e$, we derive in a similar way as \eqref{scaling-diff},
  \begin{equation*}
    \begin{split}
      &\ \frac{1}{\e^{\alpha-1}}\left( \tilde\A^\e \hat b - \tilde\A \hat b \right)(x/\e) \\
      =&\ \frac{1}{\e^{\alpha-1}} \int_{\Ro d} \left[ \hat b( x/\e+z)-\hat b(x/\e) - z\cdot\nabla \hat b(x/\e) \right] (\kappa(x/\e,\e z,z)-\tilde\kappa(x/\e, z))J(z)dz \\
       &\ + \left(c(x/\e) + \int_{B^c}z\kappa(x/\e,z,z/\e)J(z)dz\right) \cdot \nabla \hat b(x/\e) \\
      =: &\ II_1(x/\e) + II_2(x/\e).
    \end{split}
  \end{equation*}
  Denote $\gamma=(\alpha+\beta)\wedge 2>\alpha$. Since $\hat b\in \C^{\alpha+\beta}_\mu(\T^d)$, we apply Taylor expansion to get that for all $x\in\T^d$,
  \begin{equation*}
    \begin{split}
      \left| \hat b( x+z)-\hat b(x) - z\cdot\nabla \hat b(x) \right| &\le |z| \int_0^1 \left|\nabla \hat b(x+rz)-\nabla \hat b(x) \right|dr \ind_{\{|z|\le1\}} + 2\|\hat b\|_{\C^1} |z| \ind_{\{|z|>1\}} \\
      &\le \frac{1}{\gamma} \|\hat b\|_{\C^\gamma} |z|^{\gamma}\ind_{\{|z|\le1\}} + 2\|\hat b\|_{\C^1} |z| \ind_{\{|z|>1\}} \\
      &\le 2\|\hat b\|_{\C^\gamma} (|z|^{\gamma} \wedge |z|).
    \end{split}
  \end{equation*}
  Then applying assumption \eqref{ez-1} and dominated convergence, we estimate $II_1$ as follows,
  \begin{equation*}
    |II_1(X_{s}^{\epsilon}/\e)| \le \sup_{x\in\T^d} |II_1(x/\e)| \le 2j_2\|\hat b\|_{\C^\gamma} \int_{\Ro d} \left(\frac{1}{\e^{\alpha-1}} \sup_{x\in\T^d} |\kappa(x,\e z,z)-\tilde\kappa(x, z)| \right) (|z|^{\gamma} \wedge |z|) \frac{dz}{|z|^{d+\alpha}} \xrightarrow{\e\to0^+} 0.
  \end{equation*}
  Using the same argument as the proof of \eqref{B-conv}, we have the following locally uniform convergence in $t$ in probability, as $\e\to0^+$,
  \begin{equation*}
    \begin{split}
       I_2^\e(t) &\sim \int_0^t II_2 \left(\frac{X_{s}^{\epsilon}}{\epsilon}\right) ds = \int_0^t \left\{ \left[ c\left(\frac{X_{s}^{\e}}{\e}\right) + \int_{B^c}z \kappa\left( \frac{X^{\e}_s}{\e},z, \frac{z}{\e} \right)J(z)dz\right] \cdot \nabla \hat b\left(\frac{X_{s}^{\e}}{\e}\right) \right\} ds \\
       &\to t \left[ \int_{\T^d}c(x)\cdot \nabla \hat b(x) \mu(dx) + \int_{B^c}z \cdot\left( \int_{\T^d} \int_{\T^d} \nabla \hat b(x) \kappa_0(x,z,u) du \mu(dx) \right) J(z) dz \right].
    \end{split}
  \end{equation*}
  For $I_3^\e$, we use It\^o's isometry to get
  \begin{equation*}
    \begin{split}
       \E (|I_3^\e(t)|^2) &= \E \int_0^t \int_0^\infty \int_{\Ro d}\left| \epsilon\left[\hat b_\epsilon\left(X_{s-}^{\epsilon}+ \ind_{[0,\kappa(X_{s-}^\e/\e,z,z/\e))}(r)z \right)-\hat b_\epsilon\left(X_{s-}^{\epsilon}\right) \right] \right|^2 J(dz)drds \\
         &= \E \int_0^t \int_{\Ro d} \epsilon^2 \left| \hat b_\epsilon\left(X_{s-}^{\epsilon}+ z \right)-\hat b_\epsilon\left(X_{s-}^{\epsilon}\right) \right|^2 \kappa\left( \frac{X^{\e}_{s-}}{\e},z, \frac{z}{\e} \right) J(dz)ds \\
         &\le \kappa_2 j_2 t \left( 4 \|\hat b\|_\infty^2 \e^2 \int_{B_\e^c} \frac{dz}{|z|^{d+\alpha}} + \|\hat b\|_{\C^1}^2\int_{B_\e\setminus\{0\}}|z|^2 \frac{dz}{|z|^{d+\alpha}} \right) \\
         &= \kappa_2 j_2 t \omega_{d-1} \left( \frac{4 \|\hat b\|_\infty^2}{\alpha} + \frac{\|\hat b\|_{\C^1}^2}{2-\alpha} \right) \e^{2-\alpha},
    \end{split}
  \end{equation*}
  which goes to zero as $\e\to0^+$, where $\omega_{d-1}$ is the surface area of the unit sphere in $\R^d$. This implies that $I_3^\e(t)$ converges to $0$ locally uniformly in $t$ in probability. Since the local uniform topology is stronger than the Skorokhod topology in the space $\D$ (see, for instance, \cite[Proposition VI.1.17]{JS87}), $I_3^\e$ converges to $0$ in Skorokhod topology in probability and thereby in distribution. Further, it is easy to verify that the semimartingale characteristics of $I_4^\e+I_5^\e$ are $(0,0,\nu^\e)$, whose convergence is proved in \eqref{nu-conv}. Combining these convergence together and using the functional central limit theorem again, we get the results.
\end{proof}

\begin{remark}\label{remark-3}
  (i). Note that $\kappa_1\le\bar\kappa(z)\le\kappa_2$ for all $z$, so the homogenized measure $\bar \nu$ is an $\alpha$-stable L\'evy measure.

  (ii). The generator of the limit process $\bar X$, restricted to $\C^\infty(\T^d)$, is
  \begin{equation*}
    \bar\A f(x) = \int_{\Ro d} \left[ f( x+z)-f(x) - z\cdot\nabla f(x) \ind_{[1,2)}(\alpha)\ind_{\{|z|<1\}} \right] \bar\kappa(z) J(z)dz  + \bar c\cdot \nabla f(x) \ind_{(1,2)}(\alpha).
  \end{equation*}

  (iii). Note that the homogenized coefficients $\bar\kappa$ and $\bar c$ both depend on the invariant distribution $\mu$ of the auxiliary process $\tilde X$. Proposition \ref{exponential-ergodicity} tells that $\mu$ can be approximated by large-time distributions of $\tilde X$, with exponentially small error. But in practice this scheme is not efficient, since one needs to generate enormously large number of samples at a large time in order to compute the measure $\mu$ by Monte Carlo method. However, by Remark \ref{remark-2}.(i), we can approximate $\mu$ by the long-time average of a single path of $\tilde X$, due to the ergodicity. Indeed, taking $f=\ind_A$ for some $A\in\B(\T^d)$ we have
  $$\frac{1}{T} \int_0^T \ind_A(\tilde X_s) ds \to \mu(A), \quad \text{as } T\to\infty, \text{ in } L^2(\Omega,\P).$$
\end{remark}

\section{Examples and comparisons}\label{example}

In this section, we present some examples that cover several results in earlier papers.

\begin{example}[Pure jump L\'evy processes]
  In the special case that $b=c\equiv0$ and $\kappa^*(x,z,u,v)$ $\equiv\kappa^*(u)$ which is a periodic function of period 1 and satisfies $\kappa_1\le\kappa^*(u)\le\kappa_2$ for all $u$, the homogenized constant is $\bar\kappa = \int_{\T^d} \kappa^*(u) du$ and $\bar b=0$. This is the case presented in \cite[Remark 5]{SU21}. Note that in that paper, the authors use a pure analytical approach --- Mosco convergence --- to identify the limit process.
\end{example}

\vskip 0.1in
\begin{example}[SDEs with jump noise]
  Let $L^\alpha=\{L_t^\alpha\}_{t\ge 0}$ be a $d$-dimensional isotropic $\alpha$-stable L\'evy process on a filtered probability space $(\Omega,\F,\P,\{\F_t\}_{t\ge0})$ given by
  \begin{equation*}
    L_t^\alpha=\int_0^t\int_{\Bo}y\tilde N^\alpha(dy,ds)+ \int_0^t\int_{B^c}y N^\alpha(dy,ds),
  \end{equation*}
  where $1<\alpha<2$, $N^\alpha$ is a Poisson random measure on $(\Ro d)\times\R_+$ with jump intensity measure $\nu^\alpha(dy)=\frac{dy}{|y|^{d+\alpha}}$, $\tilde N^\alpha$ is the associated compensated Poisson random measure, that is, $\tilde N^\alpha(dy,ds):= N^\alpha(dy,ds)-\nu^\alpha(dy)ds$. Consider the following SDE
  \begin{equation}\label{SDE}
     \begin{split}
       X_t^{x,\epsilon} = &\ x + \int_0^t \left(\frac{1}{\epsilon^{\alpha-1}} b\left(\frac{X_{s-}^{x,\epsilon}}{\epsilon}\right)+ c\left(\frac{X_{s-}^{x,\epsilon}}{\epsilon}\right)\right)ds \\
       &\ +\int_0^t\int_{\Bo}\sigma\left( \frac{X_{s-}^{x,\epsilon}}{\epsilon},y\right) \tilde N^\alpha(dy,ds)+ \int_0^t\int_{B^c}\sigma \left(\frac{X_{s-}^{x,\epsilon}}{\epsilon},y\right) N^\alpha(dy,ds),
    \end{split}
  \end{equation}
  where the functions $b,c$ are all periodic of period 1, the function $\sigma(x,y)$ is periodic in $x$ of period 1, and odd in $y$ in the sense that $\sigma(x,-y)=-\sigma(x,y)$ for all $x,y\in\R^d$. We assume that $\sigma\in\C^{1,2}(\R^d\times\R^d)$ and there exists constants $C_1>0, C_2>1$, such that for all $x_1,x_2,x,y\in\R^d$,
  \begin{equation*}
    |\sigma(x_1,y)-\sigma(x_2,y)|\le C_1|x_1-x_2| |y|, \quad
    C_2^{-1}|y| \le |\sigma(x,y)| \le C_2|y|.
  \end{equation*}
  Assume in addition that for every $x$, $\sigma(x,\cdot)$ is uniformly continuous and is $\C^2$-diffeomorphism with inverse $\tau(x,\cdot):=\sigma(x,\cdot)^{-1}$. Then we know that \eqref{SDE} possesses a unique strong solution which is a Feller process, for each $\e>0$, see \cite[Theorem 4.2, Corollary 4.3]{HDS22}.

  Now the generator of the solution processes $X^{x,\e}$ restricted to $\C^\infty(\T^d)$ is
  \begin{equation*}
    \begin{split}
      \A^\e_\alpha f(x) :=&\ \int_{\Ro d} \left[ f\left( x+\sigma\left(\frac{x}{\e},y\right)\right)-f(x)- \sigma\left(\frac{x}{\e},y\right) \cdot\nabla f(x) \ind_{B}(y) \right] \nu^\alpha(dy)  \\
      &\ + \left[\frac{1}{\e^{\alpha-1}} b\left(\frac{x}{\e}\right) + c\left(\frac{x}{\e}\right) \right] \cdot\nabla f(x).
    \end{split}
  \end{equation*}
  We can rewrite it, by a change of variables and the oddness of $y\to\sigma(x,y)$, to the form in \eqref{Ae} with
  \begin{equation}\label{kappa-tau}
    \kappa(x,z,u)\equiv \kappa(x,z): = |\det \nabla_z\tau(x,z)| \frac{|z|^{d+\alpha}}{|\tau(x,z)|^{d+\alpha}},
  \end{equation}
  that is,
  \begin{equation}\label{kappa-sigma}
    \int_A \kappa(x,z)\frac{dz}{|z|^{d+\alpha}} = \int_{\Ro d} \ind_A(\sigma(x,y))\nu^\alpha(dy),\quad A\in\B(\Ro d).
  \end{equation}
  Then the function $\kappa$ satisfies assumptions \eqref{non-degenerate} and \eqref{Holder} (see \cite[ Assumption H3, Lemma 2.3, Proposition 2.5]{HDS22}), as well as assumption \eqref{kappa_0} with $\kappa_0(x,z,u) \equiv \kappa(x,z)$. Note that for each $x$, the oddness of $\sigma(x,\cdot)$ implies the oddness of $\tau(x,\cdot)$, and further the symmetry of $\kappa(x,\cdot)$ in the sense that
  \begin{equation*}
    \kappa(x,z)=\kappa(x,-z) \quad \text{for all }x,z.
  \end{equation*}

  We assume further that (cf. \cite[Assumption H5]{HDS22}),
  \begin{align*}
    \textstyle{\DivEps 1}\sigma(x,\e y) \to \nabla_y\sigma(x,0)\cdot y, \quad\text{uniformly in } x \text{ and } y, \quad \text{as } \e\to 0^+,
  \end{align*}
  Then we can prove easily (e.g., by \cite[Theorem 7.17]{Rud76}) that for each $z$,
  \begin{align*}
    & 
    \textstyle{ \DivEps 1} \tau(x,\e z) \to \nabla_z \tau(x,0) \cdot z \quad\text{and}\quad \nabla_z \tau(x,\e z)\to \nabla_z \tau(x,0), \quad\text{uniformly in } x, y, \quad \text{as } \e\to 0^+.
  \end{align*}
  Hence, we conclude that the function $\kappa$ defined in \eqref{kappa-tau} satisfies assumption \eqref{ez-1} with
  \begin{equation*}
    \tilde\kappa(x,z) \equiv \tilde\kappa(x):= |\det \nabla_z\tau(x,0)| \frac{1}{|\nabla_z\tau(x,0)|^{d+\alpha}},
  \end{equation*}
  Applying Theorem \ref{homogenization}, we know that the sequence of solutions $X^{x,\epsilon}$ converges in distribution to a L\'evy process $\bar X^x$ starting from $x$ with L\'evy triplet $(\bar b,0,\bar \nu)$ given in \eqref{homogenized}. By the symmetry of $\kappa$ and $\nu^\alpha$, the homogenized constant $\bar b = \int_{\T^d}(I+\nabla \hat b(x))\cdot c(x)\mu(dx)$, where $\mu$ is the invariant measure of the Feller process generated by
  \begin{equation*}
    \tilde\A_\alpha f(x) := \int_{\Ro d} \left[ f\left( x+\nabla_y\sigma\left(x,0\right)\cdot y\right)-f(x)- y\cdot \nabla_y\sigma\left(x,0\right) \cdot\nabla f(x) \ind_{B}(y) \right] \nu^\alpha(dy) + b\left(x\right)\cdot\nabla f(x),
  \end{equation*}
  and $\hat b$ is the unique solution to the Poisson equation $\tilde \A_\alpha\hat b=b$. Moreover, the homogenized function is $\bar\kappa(z) = \int_{\T^d} \kappa(x,z) \mu(dx)$. This coincides with the result in \cite[Theorem 5.2]{HDS22}. To see this, we derive $\bar\nu(A)$ for $A\in\B(\Ro d)$ by \eqref{kappa-sigma},
  $$\bar\nu(A) = \int_A \int_{\T^d} \kappa(x,z) \mu(dx) \frac{dz}{|z|^{d+\alpha}} = \int_{\Ro d}\int_{\T^d}\ind_A(\sigma(x,y) )\mu(dx)\nu^\alpha(dy).$$
  In particular, this also generalizes the result in \cite{Fra07}, where the author consider the special case $\sigma(x,y)=\sigma_0(x)y$.
\end{example}

\vskip 0.1in
\begin{example}[One-dimensional jump processes]
  Consider the one-dimensional case with $\alpha\in(1,2)$, $c\equiv0$ and $\kappa^*(x,z,u,v)\equiv \kappa^*(x,v)$, that is,
  \begin{equation*}
    \A^\e_{\mathrm{1d}} f(x) = \int_{-\infty}^{+\infty} \left[ f( x+z)-f(x) - zf'(x) \ind_{\{|z|<1\}}(z) \right] \kappa^*(\textstyle{\frac{x}{\e}}, \textstyle{\frac{z}{\e}})J(z)dz + \textstyle{\frac{1}{\e^{\alpha-1}}} b(\textstyle{\frac{x}{\e}}) f'(x).
  \end{equation*}
  Here $J$ is the density of an $\alpha$-stable L\'evy measure on $\R\setminus\{0\}$ (see \cite[Remark 14.4]{Sat99}), that is,
  \begin{equation*}
    J(z) = j^+ z^{-(1+\alpha)}\ind_{(0,+\infty)}(z) + j^- |z|^{-(1+\alpha)}\ind_{(-\infty,0)},
  \end{equation*}
  with constants $j^+,j^->0$, so that assumption \eqref{bound} is fulfilled.

  Besides assumptions \eqref{non-degenerate}, \eqref{Holder}, \ref{bc} and \ref{center}, we assume further that there exists two functions $\kappa_0^+,\kappa_0^-:\Ro d\to [0,\infty)$ such that for each $x$,
  \begin{equation*}
    \lim_{y\to\pm\infty} y^{-1} \int_0^y \kappa^*(x,v) dv = \kappa_0^{\pm}(x).
  \end{equation*}
  Note that this is the type of assumption in \cite{HIT77}. Then by L'H\^opital's rule, we have
  \begin{equation*}
    \lim_{v\to\pm\infty} \kappa^*(x,v) = \kappa_0^{\pm}(x).
  \end{equation*}
  Thus, our assumption \eqref{kappa_0} is fulfilled by letting $$\kappa_0(x,z,u)\equiv\kappa_0(x,z):= \kappa_0^+(x)\ind_{(0,+\infty)}(z)+ \kappa_0^-(x)\ind_{(-\infty,0)}(z).$$
  And assumption \eqref{ez} hold trivially.
  Now using Theorem \ref{homogenization}, the Feller process generated by $\A_{\mathrm{1d}}^\e$ converges in distribution, as $\e\to0^+$, to a one-dimensional $\alpha$-stable L\'evy process $\bar X$ with L\'evy triplet $(\bar b,0,\bar \nu)$ in \eqref{homogenized}. Let $\mu$ be the invariant measure of the Feller process generated by
  \begin{equation*}
    \tilde\A_{\mathrm{1d}} f(x) = \int_{-\infty}^{+\infty} \left[ f( x+z)-f(x) - zf'(x) \right] \kappa^*(x,z)J(z)dz + b(x) f'(x).
  \end{equation*}
  Then the homogenized drift $\bar b$ is
  \begin{equation*}
    \bar b = \frac{1}{\alpha-1} \int_0^1 \left( j^+ \kappa^+(x) + j^- \kappa^-(x) \right) \hat b'(x) \mu(dx),
  \end{equation*}
  where $\hat b$ is the unique solution to the Poisson equation $\tilde\A_{\mathrm{1d}} \hat b = b$. Define two constants  $\bar\kappa^{\pm}:=\int_{\T^d}\kappa_0^{\pm}(x)\mu(dx)$, then $$\bar\kappa(z) = \bar\kappa^+\ind_{(0,+\infty)}(z)+ \bar\kappa^-\ind_{(-\infty,0)}(z).$$
  Note that the authors in \cite{HIT77} consider the operators of the form $\tilde\A_{\mathrm{1d}}$ with $\kappa^*(\frac{x}{\e},\frac{z}{\e})$ and $\frac{1}{\e^{\alpha-1}}b(\frac{x}{\e})$ in place of $\kappa^*(x,z)$ and $b(x)$, which is slightly different from $\A^\e_{\mathrm{1d}}$, but the homogenized jump measure therein coincides with $\bar\nu$.
\end{example}

\section{Generalization to symmetric stable-like processes with variable order}\label{stable-like}

  A class of pure jump processes is of great interests, called \emph{stable-like} processes (see the survey \cite{JS01} and literature therein). Locally, a stable-like process looks like a stable process, so that for every $x$, its jump measure $\eta(x,\cdot)$ is $\pmb\alpha(x)$-stable \cite[Theorem 14.3]{Sat99}, i.e.,
  \begin{equation}\label{eta}
    \eta(x,A) = \int_0^\infty \int_S \ind_A(r\xi) \rho(x,d\xi) \frac{dr}{r^{1+\pmb{\alpha}(x)}}, \quad A\in\B(\Ro d),
  \end{equation}
  where $\rho$ is a map from $\R^d$ to the space $\mathcal M(S)$ of finite measures on $S$, called the \emph{spherical part} of $\eta$ or the \emph{spectral measure} of the process; the stability index  $\pmb{\alpha}$ is now a function taking values in $(0,2)$. Due to the variety of $\pmb{\alpha}$, such a jump kernel $\eta$ cannot be written as the product of a bounded function $\kappa$ with a reference L\'evy measure with constant stability index (cf. \eqref{non-degenerate} and \eqref{homogeneous}), so the homogenization framework in previous sections cannot be applied to such jump processes. However, we can modify a bit the assumptions for the coefficient $\kappa$ to deal with such case. Note that some authors also use the terminology ``stable-like'' to refer to the case \eqref{Ae} (or \eqref{L^{b,eta}}, e.g., \cite{CZ17}). But we shall reserve it as the case \eqref{eta} in order to distinguish.

\subsection{Homogenization result}


In order that there exist a jump process with jump kernel \eqref{eta}, we need some assumptions \cite{KKS21}:
\begin{itemize}
  \item $\pmb\alpha:\R^d\to(0,2)$ is of class $\C^1$, periodic of period 1, and satisfies that for all $x\in\R^d$,
      \begin{equation}\label{bold-alpha}
        0< \alpha := \min_{x\in\R^d}\pmb\alpha(x) \le \bar\alpha := \max_{x\in\R^d}\pmb\alpha(x)<2,
      \end{equation}
      where the minimum and maximum of $\pmb\alpha$ are attainable since $\pmb\alpha$ is continuous and periodic;
  \item $\rho: \R^d\to \mathcal M(S)$ is periodic of period 1 and symmetric, i.e., $\rho(x,\xi)=\rho(x,-\xi)$ for all $x\in\R^d$ and $\xi\in S$, and has a density, still denoted by $\rho$, i.e, $\rho(x,d\xi) = \rho(x,\xi)d\xi$,
      and satisfies the following conditions:
  \begin{itemize}
    \item[-] $\inf_{x\in\R^d} \left(\rho(x,S) \wedge \inf_{\theta\in S} \int_S (\theta\cdot\xi)^2 \rho(x,d\xi) \right) >0$,
    \item[-] $\rho$ is Lipschitz in the sense that there exists $C>0$ such that for all $x,y \in\R^d$,
    \begin{equation*}
      \left| \rho(x,S) - \rho(y,S) \right| + \mathcal W_1(\hat\rho(x,\cdot), \hat\rho(y,\cdot)) \le C|x-y|,
    \end{equation*}
    where $\hat\rho = (\rho(\cdot,S))^{-1}\rho$ is the normalized probability measure of $\rho$ and $\mathcal W_1$ is Wasserstein-1 distance of probability measures (e.g., \cite{Vil03}),
    \item[-] $\rho$ is dominated by a probability function $\rho_0$ on $S$, that is, there exists a constant $C>0$ such that $\rho(x,\xi) \le C\rho_0(\xi)$ for all $x\in\R^d, \xi\in S$.
  \end{itemize}
\end{itemize}

We still consider the operator $\A^\e$ in \eqref{Ae}, with coefficients:
\begin{itemize}
  \item $b=c\equiv0$, $J(z) = |z|^{-(d+\alpha)}$ is the density of an rotationally invariant $\alpha$-stable L\'evy measure with $\alpha$ given in \eqref{bold-alpha}, 
  \item $\kappa^*$ is given by
  \begin{equation*}
    \kappa^*(x,z,u,v)\equiv \kappa^*(x,v) := \rho(x,v/|v|) |v|^{\alpha-\pmb\alpha(x)},
  \end{equation*}
\end{itemize}
The resulting function $\kappa(x,z,u)\equiv\kappa^*(x,u)$ does not satisfying either \eqref{non-degenerate} or \eqref{Holder} in general. But \eqref{kappa_0} still holds with
$$\kappa_0(x,z,u)\equiv\kappa_0(x,z):= \rho(x,z/|z|) \ind_{\{\pmb\alpha(x)=\alpha\}},$$
and \eqref{ez} holds trivially with $\tilde\kappa = \kappa = \kappa^*$.

%

Note that because of the symmetry of $\rho$, the indicator function $\ind_{[1,2)}(\alpha)$ in \eqref{Ae} has no effect. The jump measures of $\A^\e$ is of the form \eqref{eta} with $\pmb\alpha$ and $\rho$ replaced by
  \begin{equation}\label{new-jump-coef}
    \pmb\alpha_\e(x) := \pmb\alpha(x/\e), \quad \rho^\e(x,\xi) := \e^{\pmb\alpha(x/\e)-\alpha} \rho(x/\e,\xi).
  \end{equation}
Since $\kappa$ and $\tilde\kappa$ coincide and the jump kernel $\rho$ is symmetric, we see that $\tilde \A^\e \equiv \tilde\A$ for all $\e$ (cf. \eqref{Ae-tilde} and \eqref{A-tilde}). Their jump measure is given by \eqref{eta} with $\rho(x,d\xi) = \rho(x,\xi)d\xi$.

The counterpart of Proposition \ref{heat_kernel} is the following, where the well-posedness is taken from \cite[Theorem 3.1]{KKS21} and the heat kernel estimate is adapted from \cite[Proposition 3.1, Theorem 5.1]{Kol00}.
\begin{proposition}\label{heat-kernel-stl}
  Under the above listed conditions, for every $x\in\R^d$, the martingale problems for $(\A^\e,\delta_x)$, $\e>0$ and  $(\tilde\A,\delta_x)$ have unique solutions $\P^\e_x$ on $(\D,\B(\D))$ and $\tilde\P_x$ on $(\D_{\mathrm{per}},\B(\D_{\mathrm{per}}))$, respectively. The coordinate processes $X^\e$ and $\tilde X$ are $\R^d$- and $\T^d$-valued Feller process, respectively, starting from $x$. Moreover, $\tilde X$ has a jointly continuous transition probability density $\tilde p(t;x,y)$ satisfying that for every $T>0$, there exist constants $0<C_1<1$, $C_2, C_3>0$ and $\delta\in(0,1)$ such that for all $t\in (0,T]$ and $x,y\in \T^d$,
  \begin{equation}\label{heat-kernel-est-stl}
    \begin{split}
      \tilde p(t;x,y) &\ge \sum_{l\in\Z^d} \bigg\{ C_1 \left[t^{-d/\pmb\alpha(x)}\wedge \left(t |x-y+l|^{-(d+\pmb\alpha(x))}\right)\right] \left( 1- C_2 t^\gamma \right) \\
      &\qquad\qquad - C_3 t^\delta \left[ 1 \wedge |x-y+l|^{-(d+\pmb\alpha(x))} \right] \bigg\} \vee 0,
    \end{split}
  \end{equation}
  with any $0< \gamma < 1/(d+\pmb\alpha(x))$ and $0<\delta<1- \beta(d+\pmb\alpha(x))$.
\end{proposition}
By \eqref{new-jump-coef}, we have for all $\e>0$ and `good' test function $f:\R^d\to\R$,
\begin{equation}\label{scaling-stl}
  \A^\e f(x) = \e^{-\alpha} (\tilde \A f_{1/\e})(x/\e),
\end{equation}
so that Lemma \ref{rescale} holds with $\{X^\e_t\}_{t\ge0} \stackrel{\mathtt d}{=} \big\{ \e \tilde X_{t/\e^{\alpha}} \big\}_{t\ge0}$ in this case. Proposition \ref{exponential-ergodicity} and Lemma \ref{invatiant-meausre} holds trivially with $\tilde X^\e \equiv \tilde X$, $\tilde P^\e_t \equiv \tilde P_t$ and $\mu_\e \equiv \mu$. In particular, to prove the counterpart of Proposition \ref{exponential-ergodicity}, as indicated in its own proof, it suffices to show that there exists a $t_0>0$ such that $\tilde p(t_0;x,y)$ is bounded from below by a positive constant independent of $x, y\in\T^d$. To this purpose, we choose $\gamma_0>0$ and $t_0 <1$ such that
\begin{equation*}\left\{
  \begin{aligned}
    &\gamma_0< 1/(d+\bar\alpha), \quad t_0^{1/\alpha} \ge 1/\sqrt{2}, \\
    &1- C_2 t_0^{\gamma_0} >0, \quad C_1 t_0^{-d/\bar\alpha} (1- C_2 t_0^{\gamma_0}) - C_3 >0.
  \end{aligned}\right.
\end{equation*}
Since for any $x,y\in\T^d = [0,1]^d$, there is always a $l\in\Z^d$ such that $|x-y+l|\le 1/\sqrt{2} \le t_0^{1/\alpha} \le t_0^{1/\pmb\alpha(x)} <1$, we obtain from \eqref{heat-kernel-est-stl} that
\begin{equation*}
  \tilde p(t;x,y) \ge C_1 t_0^{-d/\pmb\alpha(x)} \left( 1- C_2 t^{\gamma_0} \right) - C_3 t_0^\delta \ge C_1 t_0^{-d/\bar\alpha} (1- C_2 t_0^{\gamma_0}) - C_3,
\end{equation*}
where the last quantity is positive and independent of $x,y$. This proves that Proposition \ref{exponential-ergodicity} holds true for the case here. As a consequence, Corollary \ref{ergodic} also holds. Therefore, the part (i) of the proof of Theorem \ref{homogenization} can still proceed with no obstacles. In conclusion, we get the following homogenization result for stable-like processes which recovers the result of \cite[Theorem 1]{Fra06}.
\begin{theorem}
  Under the same assumption as Proposition \ref{heat-kernel-stl}, 
  we have the following weak convergence on the space $\D(\R_+;\R^d)$:
  $$X^\e \ \Rightarrow \bar X, \quad \text{as } \e\to 0^+,$$
  where the limit process $\bar X$ is a L\'evy process with L\'evy triplet $(0,0,\bar \nu)$ given by
  \begin{equation*}
    \bar\nu(A) = \int_0^\infty \int_S \ind_A(r\xi) \left( \int_{\T^d} \rho(x,\xi) \ind_{\{\pmb\alpha(x)=\alpha\}} \mu(dx) \right) d\xi \frac{dr}{r^{1+\alpha}}, \quad A\in\B(\Ro d),
  \end{equation*}
  where $\mu$ is the invariant probability measure of $\tilde X$ with jump measure \eqref{eta}.
\end{theorem}
\begin{remark}
  In the special case that $\mu(\{x\in\T^d: \pmb\alpha(x)=\alpha\}) =0$, the homogenized measure $\bar{\nu}$ is trivially zero, i.e., $X^\epsilon$ converges weakly to the constant zero process. In this sense the scaling \eqref{scaling-stl} is strong. One may expect that it would give a non-trivial limit if we change the scaling \eqref{scaling-stl} to
  \begin{equation*}
    \A^\e f(x) = \big[ \e^{-\pmb\alpha} (\tilde \A f_{1/\e}) \big](x/\e).
  \end{equation*}
  But the latter cannot yield a scaling for the generated processes $X^\e$ and $\tilde X$ like Lemma \ref{rescale}. Thus, the functional convergence in Corollary \ref{ergodic} and thereby the final homogenization result are unknown in this case.
\end{remark}

\subsection{Numerical simulations}

In this subsection, we will present a numerical experiment to give some visualization for the homogenization result. Besides, for a jump particle in a periodic structure, a typical question in practical applications is the distribution property of the first exit time for the particle to escape a given domain. We will also give some visualization for the empirical mean of the first exit time.

\subsubsection{Numerical scheme}

Set the dimension $d=2$, and
\begin{equation*}
  \begin{split}
    \pmb\alpha(x) &= 1+ \frac{1}{4}\Big[ \ind_{[0,\frac{3}{8})}(x_1)\cos\left( \textstyle{\frac{8\pi}{3}} x_1 \right) + \ind_{(\frac{5}{8},1]}(x_1)\cos\left( \textstyle{\frac{8\pi}{3}} (1-x_1) \right) -\ind_{[\frac{3}{8},\frac{5}{8}]}(x_1) \\
    &\qquad\qquad + \ind_{[0,\frac{3}{8})}(x_2)\cos\left( \textstyle{\frac{8\pi}{3}} x_2 \right) + \ind_{(\frac{5}{8},1]}(x_2)\cos\left( \textstyle{\frac{8\pi}{3}} (1-x_2) \right) -\ind_{[\frac{3}{8},\frac{5}{8}]}(x_2) \Big]
  \end{split}
\end{equation*}
for $x=(x_1,x_2)\in [0,1]^2$, and
\begin{equation*}
  \rho(x,d\xi) \equiv \rho(d\xi) := \sum_{i=1}^4 \delta_{e_i}(d\xi), \quad \xi\in\S^1.
\end{equation*}
where $e_1=(1,0)$, $e_2=(0,1)$, $e_3=(-1,0)$, $e_1=(0,-1)$ are canonical orthonormal basis for $\R^2$. It is easy to verify that such $\pmb\alpha$ and $\rho$ satisfy all conditions listed in the beginning of last subsection. The minimum of $\pmb\alpha$ is $\alpha = \frac{1}{2}$. The spectral measure of the process $X^\e$ is $\rho^\e(x,d\xi) = \e^{\alpha(x/\e)-\frac{1}{2}} \rho(d\xi)$, while the spectral measure of the limit process $\bar X$ is $\bar\rho(d\xi) = \mu\left(\pmb\alpha^{-1}(\frac{1}{2})\right) \rho(d\xi)$.

We use the method in \cite{MN94} to simulate the `one-step' stable random vectors, and then use the scheme developed in \cite{Bot10} to simulate the stable-like processes $X^\e$ and $\tilde X$ by gluing all `one-step' stable random vectors together. The convergence of the latter scheme is proved in \cite{BS09}. Note that the distribution of each `one-step' stable random vector depends on the position of the previous step. As for the limit L\'evy process $\bar X$, the `one-step' vectors are independent of the previous positions.

In all path sampling, we always use time-step size $\Delta t=0.01$. There are two ways to approximately compute $\mu\left(\pmb\alpha^{-1}(\frac{1}{2})\right)$, as mentioned in Remark \ref{remark-3}.(iii). We first generate 1000 sample paths of the process $\tilde X$ with 100 steps by the above-mentioned scheme, and count the number of samples at the final time step inside the set $\pmb\alpha^{-1}(\frac{1}{2}) = [\frac{3}{8},\frac{5}{8}]^2$. Then we use a sample path with time length $T=100$, and calculate the time-average $\frac{1}{t} \int_0^t \ind_{\pmb\alpha^{-1}(\frac{1}{2})} (\tilde X_s) ds$ for varying $t$. The following figure (Fig.~\ref{inv_meas}) shows the results by these two methods. It shows that when $t$ is large, the two results are very close.

\begin{figure}[!thb]
  \centering
  \includegraphics[width=0.6\textwidth]{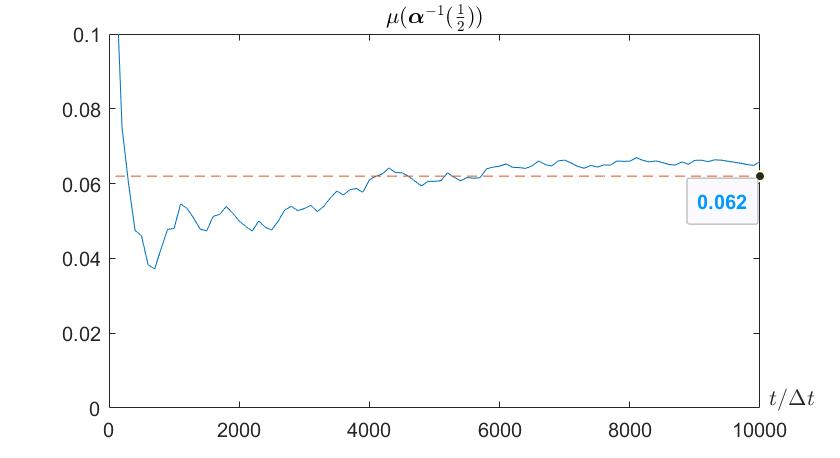}
  \caption{Computations of $\mu\left(\pmb\alpha^{-1}(\frac{1}{2})\right)$. The horizontal coordinate indicates the number of step $t/\Delta t$ and the vertical coordinate indicate the time-average of times of a single path inside $\pmb\alpha^{-1}(\frac{1}{2})$, with time-step size $\Delta t=0.01$. The dashed line is the result of Monte Carlo method, by generating 1000 samples with 100 steps.} \label{inv_meas}
\end{figure}

The following figure (Fig.~\ref{sample_path}) shows the sample paths on the plane for the processes $X^\e$ with $\e=10^{-1},10^{-2},10^{-3},10^{-4},10^{-5}$ and the limit process $\bar X$. As we can see, when the scaling parameter $\e$ gets smaller and smaller, the path of $X^\e$ is getting more and more concentrated into some small clusters.

\begin{figure}[!thb]
  \centering
  \includegraphics[width=1\textwidth]{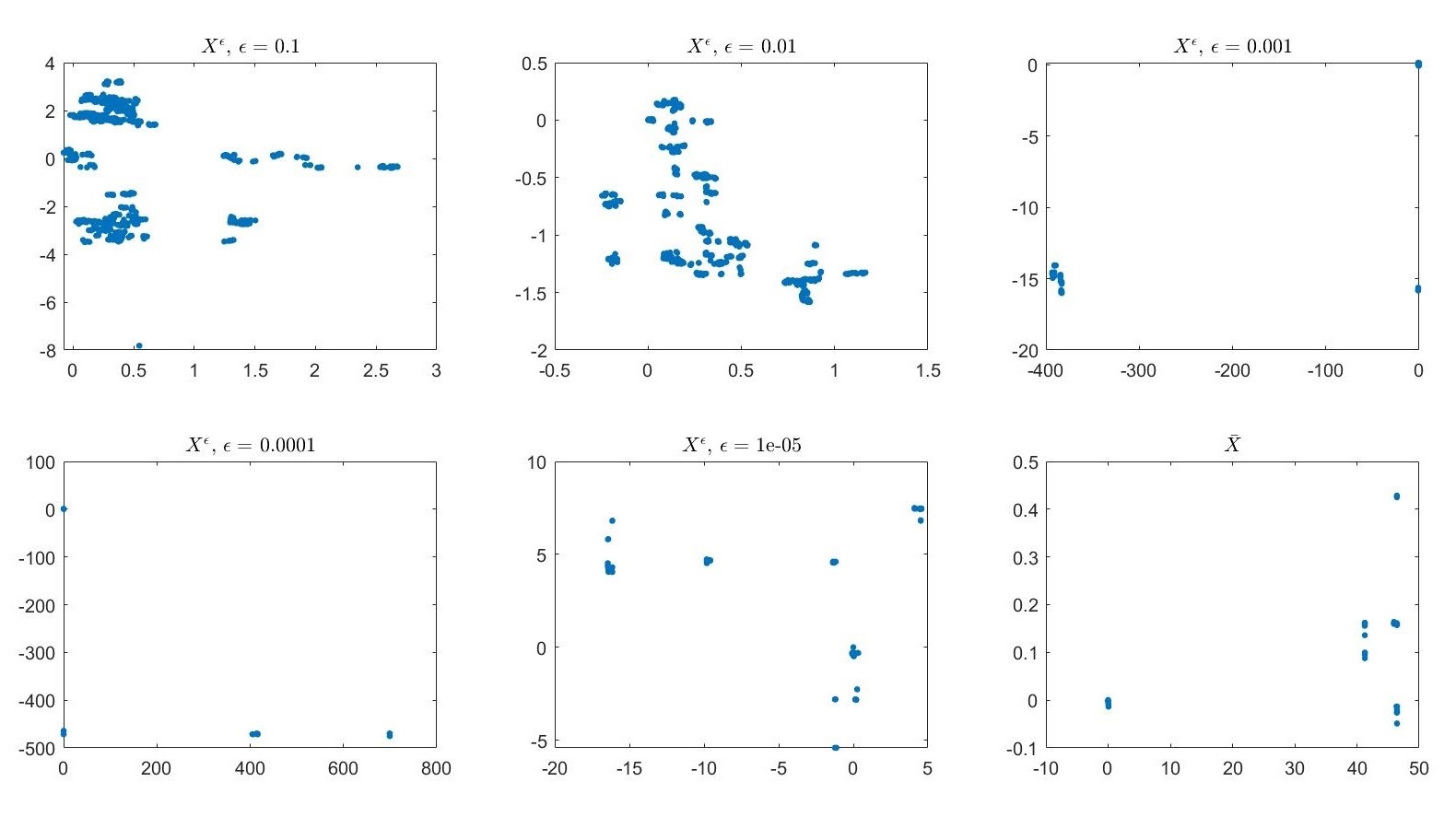}
  \caption{Sample paths for $X^\e$ and $\bar X$ on the time interval $[0,10]$ with time-step size $0.01$. The coordinates represent the particle positions in $\R^2$.} \label{sample_path}
\end{figure}

\subsubsection{Simulations of first exit time}

For $x\in\D$ and $r>0$, define
\begin{align*}
  S_r(x) &:= \inf\{t\ge0: |x(t)|\ge r \text{ or } |x(t-)|\ge r\}, \\
  S_{r+}(x) &:= \inf\{t\ge0: |x(t)|> r \text{ or } |x(t-)|> r\}, \\
  V(x) &:= \{r>0: S_r(x)<S_{r+}(x)\}.
\end{align*}
It is easy to see that
\begin{equation*}
  S_r(x) = \inf\left\{t\ge0: \sup_{0\le s\le t}|x(s)|\ge r \right\},
\end{equation*}
which is exact the \emph{first exit time} for the path $x$ to escape the ball of radius $r$. In order to simplify the notation as before, we denote $X^0 := \bar X$. Using \cite[Lemma VI.2.10]{JS87}, we know that for all $\e\ge0$ and  $\omega\in\Omega$, $V(X^\e(\omega))$ is an at most countable subset of $\R_+$. It follows that each set
\begin{equation*}
  U^\e = \{ r>0: \P(r\in V(X^\e) ) = 0 \},
\end{equation*}
has full measure in $\R_+$, and thus we have
\begin{lemma}
  The set $\cap_{\e>0} U^\e$ also has full measure in $\R_+$.
\end{lemma}

Now for each $r\in \cap_{\e>0} U^\e$, the mapping $X^\e \mapsto S_r(X^\e)$ is continuous for all $\e\ge0$, by virtue of \cite[Proposition VI.2.11]{JS87}. Hence, by the continuous mapping theorem (see, e.g., \cite[Corollary 3.1.9]{EK09}), we have the following corollary, of which the second statement follows from \cite[Theorem 25.12]{Bil86}.
\begin{corollary}
  For each $r\in \cap_{\e>0} U^\e$, $S_r(X^\e)\Rightarrow S_r(\bar X)$ as $\e\to0^+$. If in addition, the family $\{S_r(X^\e)\}_{\e>0}$ is uniformly integrable, then $\E(S_r(X^\e))\to \E(S_r(\bar X))$.
\end{corollary}


We choose $r=\pi$. The following figure (Fig.~\ref{mean_exit}) shows the empirical mean of the first exit time to escape the ball of radius $\pi$ for the processes $X^\e$, $\e=10^{-6},10^{-12},10^{-18},10^{-24},10^{-30}$, and the limit process $\bar X$. Through this figure, we can see that as $\e$ gets smaller, the convergence rate of the empirical mean with respect to the number of samples is getting slower.

\begin{figure}[!thb]
  \centering
  \includegraphics[width=1\textwidth]{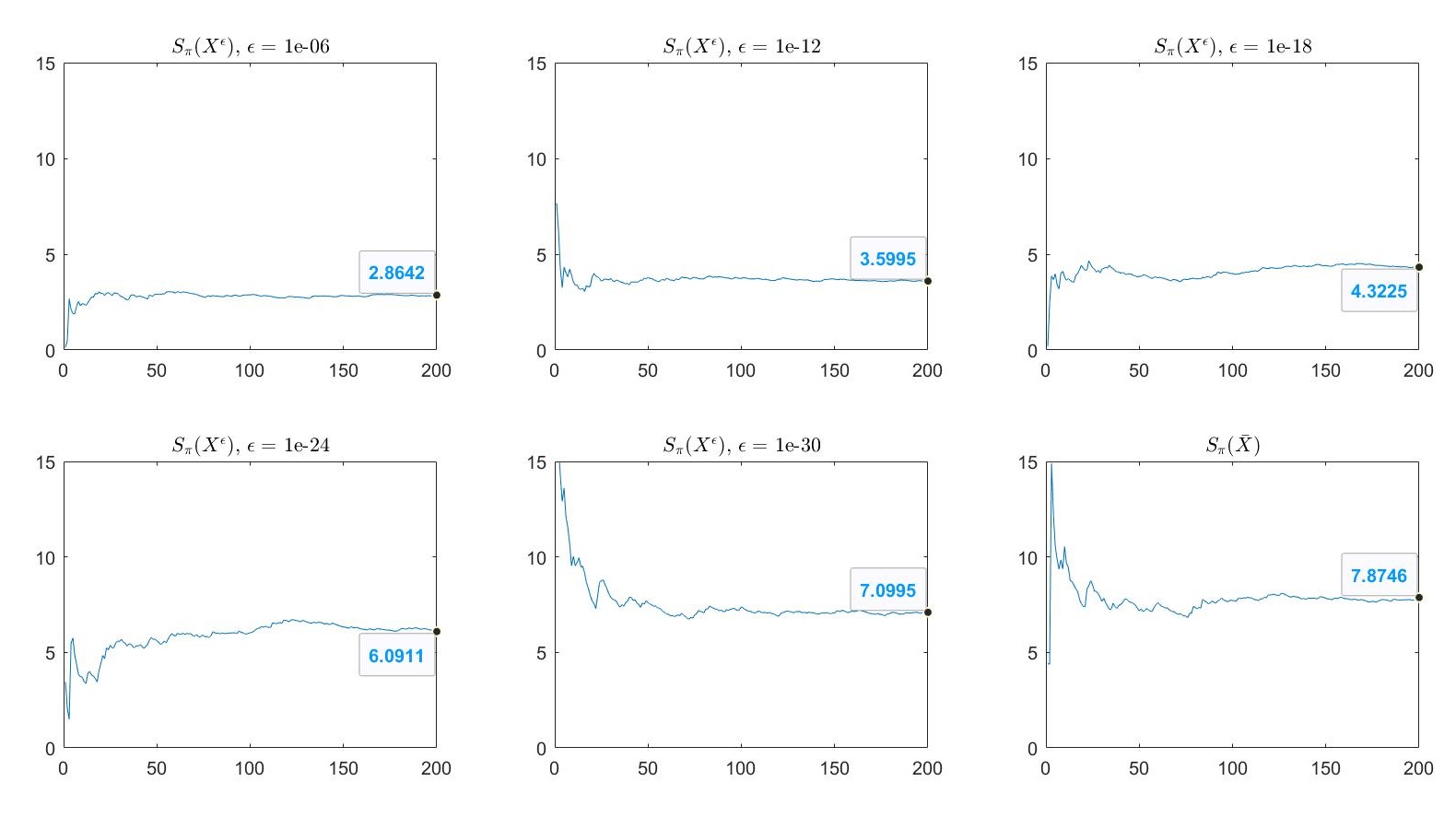}
  \caption{Empirical mean of the first exit time for $X^\e$ and $\bar X$. The horizontal and vertical coordinates indicate the number of test samples and the average to the present, respectively. The labeled points is the values of empirical mean of 200 samples with time-step size 0.01.} \label{mean_exit}
\end{figure}

The next figure (Fig.~\ref{mean_exit_conv}) shows the trend of the empirical mean of $S_\pi(X^{\e})$, $\e=10^{-n}$, with respect to $n$. It follows that the difference of empirical mean of $S_\pi(X^{\e})$ with that of $S_\pi(\bar X)$ is almost inversely proportional to $n$, so as to be proportional to $\frac{1}{\log(\e^{-1})}$. Even though this rate is not strict, we can still conclude from the figure that the convergence rate of the mean first exit time with respect to $\e$ is very slow.

\begin{figure}[!thb]
  \centering
  \includegraphics[width=0.6\textwidth]{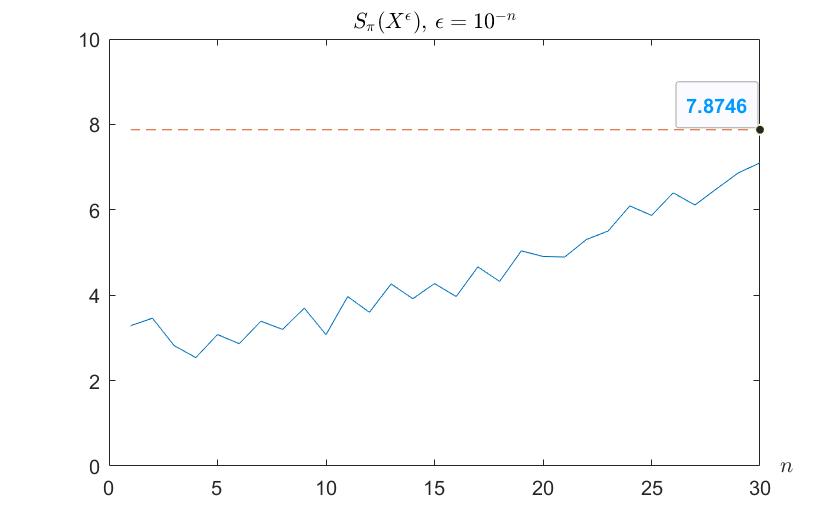}
  \caption{Empirical mean of the first exit time for $X^{\e}$ with $\e = 10^{-n}$, $n=1,2,\cdots,30$. The horizontal coordinate indicates the parameter $n$ and the vertical coordinate indicate the empirical mean of $S_\pi(X^{\e})$. The dashed line is the reference of empirical mean of $S_\pi(\bar X)$. All empirical mean are simulated with 200 samples with time-step size 0.01.} \label{mean_exit_conv}
\end{figure}

Therefore, the way of simulating the first exit time of a particle in a periodic structure by choosing a very small $\e$ is quite expensive (in computational time) and not precise in general. The advantage of our homogenization result is that we can use directly the limit process we have just identified to study its distribution properties, instead of using approximations.

%
%
%

%

\begin{appendices}

\section{Properties of the semigroups and generators}

As we have seen in Section \ref{Kernel&Ergodicity}, we need the Feller nature of the semigroups and the properties of the generators, in order to study the ergodicity of the canonical Feller processes, see Proposition \ref{heat_kernel} and Proposition \ref{exponential-ergodicity}. We devote this section to investigate the semigroups and generators. As corollaries, we can also obtain the solvability of the Possion equations with zeroth-order terms and the generalized It\^o's formula, which are used in Corollary \ref{Poisson} and the proof of our main result Theorem \ref{homogenization}.

We consider the following operator
\begin{equation}\label{L^{b,eta}}
  \L^{b,\eta} f(x) := b(x)\cdot\nabla f(x) +\int_{\Ro d} \left[ f( x+z)-f(x)- z\cdot\nabla f(x)\ind_B(z) \right] \kappa^\sharp(x,z)J(z)dz,
\end{equation}
with $\eta(x,dz):=\kappa^\sharp(x,z)J(z)dz$. We suppose that the vector field $b:\R^d\to\R^d$ is in the H\"older class $\C^\beta$ with some $\beta\in(0,1)$ and periodic of period 1, satisfying that for all $x\in\R^d$,
\begin{equation*}
  |b(x)| \le b_0,
\end{equation*}
for some constant $b_0>0$. Suppose that $J$ satisfies \eqref{j}, that is,
\begin{equation*}
  j_1 |z|^{-(d+\alpha)} \le J(z) \le j_2 |z|^{-(d+\alpha)}, \quad z\in\Ro d,
\end{equation*}
with $\alpha\in(1,2)$, $\kappa^\sharp(x,z)$ is periodic in $x$ of period 1 and satisfies similar conditions as \eqref{non-degenerate} and \eqref{Holder}, that is, for all $x,x_1,x_2,z\in\R^d$,
\begin{gather*}
  \kappa_1\le\kappa^\sharp(x,z)\le\kappa_2, \\
  |\kappa^\sharp(x_1,z)-\kappa^\sharp(x_2,z)|\le \kappa_3|x_1-x_2|^\beta.
\end{gather*}

Note that the operators $\tilde\A$, $\A^\e$ and $\tilde\A^\e$ are all of the form \eqref{L^{b,eta}} by choosing appropriate $\kappa^\sharp$. It is easy to verify that $\L^{b,\eta}f\in\C(\T^d)$ for each $f\in\C^{1+\gamma}(\T^d)$ with $1+\gamma>\alpha$.
Now we treat $\L^{b,\eta}$ as a perturbation of $\L^\eta:=\L^{0,\eta}$ by the gradient operator $\L^b:=\L^{b,0}=b\cdot\nabla$, and follow \cite{BJ07,CH15,GS19} to investigate the heat kernel for $\L^{b,\eta}$.

We introduce the following functions on $(0,\infty)\times\R^d$ for later use:
\begin{equation*}
  \varrho_\gamma(t;x) := t^{\gamma/\alpha} \left( t^{-(d+\alpha)/\alpha } \wedge |x|^{-(d+\alpha)} \right), \quad\gamma\in\R.
\end{equation*}
For abbreviation we write
$c_0$ for the set of constants $(d,\alpha,\beta,\kappa_1,\kappa_2,\kappa_3,j_1,j_2)$. Before investigating the semigroups generated by $\L^{b,\eta}$, we need some facts for the heat kernels of $\L^{\eta}$ and $\L^{b,\eta}$.

By virtue of the periodicity assumptions on the coefficients, we can choose the underlying space to be $\T^d$ instead of $\R^d$ (cf. \cite[Section 3.3.2]{BLP78}). Indeed, if $\mathfrak q^\eta(t;x,y):[0,\infty)\times\R^d\times\R^d\to\R$ is the fundamental solution of $\L^{\eta}$, then for any test function $f\in\C^\infty{(\R^d)}$ that is periodic of period 1, the function $u(t,x):=\int_{\R^d} f(y)\mathfrak q^\eta(t;x,y) dy$ must be periodic in $x$, due to Kolmogorov backward equation $\frac{\partial u}{\partial t} + \L^\eta u = 0$ and the periodicity of its initial value $u(0,x) = f(x)$ and of all coefficients in $\L^\eta$. Now we define $q^\eta(t;x,y):=\sum_{l\in\Z^d}\mathfrak q^\eta(t;x,y+l)$, then $q^\eta$ is periodic in $y$ and $u(t,x)=\int_{\T^d} f(y) q^\eta(t;x,y) dy$. Therefore, we can restrict $q^\eta$ as a function from $[0,\infty)\times\T^d\times\T^d$ to $\R$, which is exactly the fundamental solution of $\L^{\eta}$ on the state space $\T^d$. The same arguments hold for the operator $\L^{b,\eta}$. Keeping these in mind, the following facts for the operator $\L^{\eta}$ are adapted from \cite[Theorem 1.1, Theorem 1.2, Theorem 1.3, Theorem 1.4, Remark 1.5, Lemma 3.17]{GS19}.
\begin{proposition}\label{q^eta}
  (i). The fundamental solution $q^\eta(t;x,y):[0,\infty)\times\T^d\times\T^d\to\R$ of $\L^\eta$ has the following properties: for all $(t,y)\in(0,\infty)\times\T^d$, the function $x\to q^\eta(t;x,y)$ is differentiable and the derivative $\nabla_x q^\eta(t;x,y)$ is jointly continuous on $(0,\infty)\times\T^d\times\T^d$; the integral in $\L^{\eta}_x q^\eta(t;x,y)$ is absolutely integrable and the function $\L^{\eta}_x q^\eta(t;x,y)$ is jointly continuous on $(0,\infty)\times\T^d\times\T^d$. For every $T>0$, there exist a constant $C_1 = C_1(c_0,T)>0$, such that for all $t\in (0,T]$ and $x,y\in \T^d$,
  \begin{align}
    q^\eta(t;x,y) &\le C_1\sum_{l\in\Z^d}\varrho_\alpha(t;x-y+l), \label{kernel} \\
    |\nabla_x q^\eta(t;x,y)| &\le C_1\sum_{l\in\Z^d}\varrho_{\alpha-1}(t;x-y+l), \label{kernel-derivative} \\
    |\L^{\eta}_x q^\eta(t;x,y)| &\le C_1\sum_{l\in\Z^d}\varrho_0(t;x-y+l); \label{kernel-frac-derivative}
  \end{align}
  there exist $T_0 = T_0(c_0)>0$ and $C_2 = C_2(c_0)>0$ such that for all $t\in(0,T_0]$  and $x,y\in \T^d$,
  \begin{equation}\label{kernel-lower}
    q^\eta(t;x,y) \ge C_2\sum_{l\in\Z^d}\varrho_\alpha(t;x-y+l).
  \end{equation}

  (ii). Define a family of operators by
  \begin{equation}\label{T-eta}
    T^\eta_t f(x) = \int_{\T^d} f(y) q^\eta(t;x,y) dy, \quad f\in \C(\T^d),
  \end{equation}
  then $\{T^\eta_t\}_{t\ge0}$ forms a Feller semigroup on the Banach space $(\C(\T^d),\|\cdot\|_\infty)$ with generator the closure of $(\L^{\eta},\C^\infty(\T^d))$. The domain of the generator contains $\C^{1+\gamma}(\T^d)$ with $1+\gamma>\alpha$, on which the restriction of the generator is $\L^{\eta}$.
\end{proposition}
Note that the joint continuity of $\nabla_x q^\eta(t;x,y)$ is not mentioned explicitly in the previous references, but it is a consequence of \cite[Lemma 3.1, Lemma 3.5, Theorem 3.7, Lemma 3.10, Eq. (59)]{GS19}. In addition, it is only shown in the above reference that $\C^2(\T^d)$ is contained in the domain of the generator, but we can easily generalize to our case, using the same argument as the proofs of \cite[Theorem 1.3.(3a), Proposition 4.9]{GS19} and the fact that $\L^{\eta}f\in\C(\T^d)$ for each $f\in\C^{1+\gamma}(\T^d)$ with $1+\gamma>\alpha$.

For notational simplicity, the summation over the lattice $\Z^d$ will be omitted in all coming results. Keep in mind that there will be a summation over $\Z^d$ whenever the letter $l$ involved in the expression without ambiguity.

The following facts for the heat kernel of $\L^{b,\eta}$ are adapted from \cite[Theorem 1.5]{CZ18}, where the authors omitted the proof and pointed out that they can be similarly proved as in \cite{BJ07}. Since the two-sided estimates of the heat kernel of $\L^{b,\eta}$ are important for later use and also for the main part of the paper, we will only elaborate their proof.
\begin{proposition}\label{q-properties}
  There is a unique function $q^{b,\eta}(t;x,y)$ which is jointly continuous on $(0,\infty)\times\T^d\times\T^d$ and solves the following variation of parameters formula (or Duhamel's formula)
  \begin{equation}\label{integral-eqn}
    q^{b,\eta}(t;x,y) = q^\eta(t;x,y) + \int_0^t\int_{\T^d} q^{b,\eta}(t-s;x,z)b(z)\cdot\nabla_z q^\eta(s;z,y) dz ds,
  \end{equation}
  and satisfying that for every $T>0$, there is a constant $C = C(c_0,T,b_0)>0$ such that on $(0,T]\times\T^d\times\T^d$,
  $$|q^{b,\eta}(t;x,y)| \le C\varrho_\alpha(t;x-y+l).$$
  Moreover, $q^{b,\eta}$ enjoys the following properties. \\
  (i) (Conservativeness). For all $t>0$, $x\in\T^d$, $\int_{\T^d} q^{b,\eta}(t;x,y)dy=1$. \\
  (ii) (Chapman-Kolmogorov equation). For all $s,t>0$, $x,y\in\T^d$,
  \begin{equation*}
    \int_{\T^d} q^{b,\eta}(t;x,z) q^{b,\eta}(s;z,y)dz = q^{b,\eta}(t+s;x,y).
  \end{equation*}
  (iii) (Two-sided estimates). For every $T>0$, there is a constant $C_3 = C_3(c_0,T, b_0)>1$, such that on $(0,T]\times\T^d\times\T^d$,
  \begin{equation*}
    C_3^{-1} \varrho_\alpha(t;x-y+l) \le q^{b,\eta}(t;x,y) \le C_3 \varrho_\alpha(t;x-y+l).
  \end{equation*}
  (iv) (Gradient estimate). The function $\nabla_x q^{b,\eta}(t;x,y)$ is jointly continuous on $(0,\infty)\times\T^d\times\T^d$. For every $T>0$, there is a constant $C_4 = C_4(c_0,T,b_0) >0$ such that on $(0,T]\times\T^d\times\T^d$,
  \begin{equation}\label{q-b&eta-derivative}
    |\nabla_x q^{b,\eta}(t;x,y)| \le C_4\varrho_{\alpha-1}(t;x-y+l).
  \end{equation}
\end{proposition}

\begin{proof}
  We follow the lines of \cite[Theorem 2, Lemma 15]{BJ07} to prove (iii). We define a sequence of functions $\left\{q^\eta_n: (0,\infty)\times\T^d\times\T^d \mid n\in\N \right\}$ recursively by
\begin{align*}
  q^{(0)}(t;x,y) & := q^\eta(t;x,y), \\
  q^{(n+1)}(t;x,y) & := \int_0^t\int_{\T^d} q^{(n)}(t-s;x,z)b(z)\cdot\nabla_z q^\eta(s;z,y) dz ds, \quad n\in \N.
\end{align*}
  By \eqref{kernel}, \eqref{kernel-derivative} and \cite[Eq. (92), Lemma 5.17.(c)]{GS19}, we have
  \begin{equation*}
    \begin{split}
       |q^{(1)}(t;x,y)| \le&\ \int_0^t\int_{\T^d} q^\eta(t-s;x,z)\left| b(z)\cdot \nabla_z q^\eta(s;z,y) \right| dz ds \\
       \le &\ C_1^2\|b\|_\infty \int_0^t\int_{\R^d} \varrho_{\alpha}(t-s;x-z+l) \varrho_{\alpha-1}(t;z-y) dz ds \\
       \le &\ C_1^2\|b\|_\infty \textstyle{B(\frac{\alpha}{2},\frac{\alpha-1}{2} )} \varrho_{2\alpha-1}(s;x-y+l) \\
       \le &\ C_1^2\|b\|_\infty \textstyle{B(\frac{\alpha}{2},\frac{\alpha-1}{2} )}  t^{\frac{\alpha-1}{\alpha}} \varrho_{\alpha}(t;x-y+l) \\
       =: &\ c_1(c_0,T,t,b_0) \varrho_{\alpha}(t;x-y+l).
    \end{split}
  \end{equation*}
  Note that the positive constant $c_1$ is increasing in $t$ since $\alpha\in(1,2)$. By iteration and \eqref{kernel-lower}, we obtain
  \begin{equation*}
    |q^{(n)}(t;x,y)| \le \left[ c_1(c_0,T,t,b_0) \right]^n \varrho_{\alpha}(t;x-y+l) \le \left[ c_1(c_0,T,t,b_0) \right]^n C_2^{-1} q^\eta(t;x,y).
  \end{equation*}
  Choose $t_0\le T_0$ small enough such that $c_1(c_0,T,t_0,b_0) \le \frac{C_2}{1+C_2}$ and $T=n_0t_0$ for some $n_0\in\N_+$. Then for all $(t,x,y) \in (0,t_0]\times\T^d\times\T^d$,
  \begin{equation*}
    \begin{split}
      \left( 1 - \frac{C_2^{-1} c_1(c_0,T,t_0,b_0)}{1- c_1(c_0,T,t_0,b_0)} \right) q^\eta(t;x,y) &\le q^{(0)}(t;x,y) - \sum_{n=1}^{\infty} |q^{(n)}(t;x,y)| \\
      &\le \sum_{n=0}^{\infty} q^{(n)}(t;x,y) \\
      &\le \sum_{n=0}^{\infty} |q^{(n)}(t;x,y)| \le \frac{C_2^{-1}}{1- c_1(c_0,T,t_0,b_0)} q^\eta(t;x,y).
    \end{split}
  \end{equation*}
  Set
  \begin{equation*}
    c_2(c_0,T,t_0,b_0) := \left( 1 - \frac{C_2^{-1} c_1(c_0,T,t_0,b_0)}{1- c_1(c_0,T,t_0,b_0)} \right)^{-1} \vee \frac{C_2^{-1}}{1- c_1(c_0,T,t_0,b_0)}.
  \end{equation*}
  A similar argument of \cite[Section 3]{BJ07} yields that the series $\sum_{n=0}^{\infty} q^{(n)}$ converges on $(0,t_0]\times\T^d\times\T^d$ to $q^{b,\eta}$. So we get that for all $(t,x,y) \in (0,t_0]\times\T^d\times\T^d$,
  \begin{equation*}
    (c_2(c_0,T,t_0,b_0))^{-1} q^\eta(t;x,y) \le q^{b,\eta}(t;x,y) \le c_2(c_0,T,t_0,b_0) q^\eta(t;x,y).
  \end{equation*}
  Now we apply (ii) and \eqref{kernel-lower} to deduce that for any $t\in(0,T]$ and $x,y\in\T^d$,
  \begin{equation*}
    \begin{split}
       q^{b,\eta}(t;x,y) & = \int_{\T^d} \cdots \int_{\T^d} \int_{\T^d} q^{b,\eta}(\textstyle{\frac{t}{n_0}};x,x_1) q^{b,\eta}(\textstyle{\frac{t}{n_0}};x_1,x_2) \cdots q^{b,\eta}(\textstyle{\frac{t}{n_0}};x_{n_0-1},y) dx^{(n_0-1)} \\
         & \ge c_2^{-n_0} \int_{\T^d} \cdots \int_{\T^d} \int_{\T^d} q^{\eta}(\textstyle{\frac{t}{n_0}};x,x_1) q^{\eta}(\textstyle{\frac{t}{n_0}};x_1,x_2) \cdots q^{\eta}(\textstyle{\frac{t}{n_0}};x_{n_0-1},y) dx^{(n_0-1)} \\
         & = c_2^{-n_0} q^{\eta}(t;x,y) \ge c_2^{-T/t_0} C_2 \varrho_\alpha(t;x-y+l),
    \end{split}
  \end{equation*}
  where $dx^{(n_0-1)} := dx_1 dx_2\cdots dx_{n_0-1}$, and similarly,
  \begin{equation*}
    q^{b,\eta}(t;x,y)\le c_2^{T/t_0} C_1 \varrho_\alpha(t;x-y+l).
  \end{equation*}
  The result (iii) follows by taking $C_3(c_0,T, b_0) = (c_2(c_0,T,t_0,b_0))^{T/t_0} [C_1(c_0,T) \vee (C_2(c_0))^{-1}] >1$.
\end{proof}

\begin{corollary}
  The following version of variation of parameters formula holds,
  \begin{equation}\label{integral-eqn-1}
    q^{b,\eta}(t;x,y) = q^\eta(t;x,y) + \int_0^t\int_{\T^d} q^\eta(t-s;x,z)b(z)\cdot\nabla_z q^{b,\eta}(s;z,y) dz ds.
  \end{equation}
  The function $\L^{b,\eta}_x q^{b,\eta}(t;x,y)$ is jointly continuous on $(0,\infty)\times\T^d\times\T^d$. For every $T>0$, there is a constant $C_5 = C_5(c_0,T,b_0) >0$ such that on $(0,T]\times\T^d\times\T^d$,
  \begin{equation}\label{q-b&eta-frac-derivative}
    |\L^{b,\eta}_x q^{b,\eta}(t;x,y)| \le C_5 \varrho_0(t;x-y+l).
  \end{equation}
\end{corollary}

\begin{proof}
  The formula \eqref{integral-eqn-1} follows from a similar argument as the proof of \eqref{integral-eqn}, cf. \cite[Theorem 4.2]{CH15}. We prove \eqref{q-b&eta-frac-derivative}. Recall that $\L^{b,\eta} = \L^{\eta} + b\cdot\nabla$ and $\alpha>1$. By \eqref{kernel-derivative}, \eqref{kernel-frac-derivative} and \cite[Eq. (92)]{GS19}, we have for all $(t,x,y)\in (0,T]\times\T^d\times\T^d$,
  \begin{equation*}
    |\L^{b,\eta}_x q^{\eta}(t;x,y)| \le C(c_0,T,b_0) \varrho_0(t;x-y+l).
  \end{equation*}
  It follows from \eqref{q-b&eta-derivative} and \cite[Eq. (92), Lemma 5.17.(c)]{GS19} that
  \begin{equation*}
    \begin{split}
       &\ \int_0^t\int_{\T^d} \left| \L^{b,\eta}_x q^\eta(t-s;x,z)b(z)\cdot\nabla_z q^{b,\eta}(s;z,y) \right|dz ds \\
       \le &\ C(c_0,T,b_0) \int_0^t\int_{\R^d} (t-s)\varrho_{-\alpha}(t-s;x-z+l) s \varrho_{-1}(s;z-y) dz ds \\
       \le &\ C(c_0,T,b_0) \textstyle{B(1-\frac{\alpha}{2},\frac{1}{2})} \varrho_{\alpha-1}(t;x-y+l) \\
       \le &\ C(c_0,T,b_0) \varrho_0(t;x-y+l),
    \end{split}
  \end{equation*}
  where $B$ is the beta function. Combining these estimates with \eqref{integral-eqn-1}, we get \eqref{q-b&eta-frac-derivative}. The joint continuity of $\L^{b,\eta}_x q^{b,\eta}(t;x,y)$ follows from the jointly continuity of $\L^{\eta}_x q^{\eta}(t;x,y)$, $\nabla_x q^{\eta}(t;x,y)$ and $\nabla_x q^{b,\eta}(t;x,y)$ and \eqref{integral-eqn-1}.
\end{proof}

Define a family of operators
\begin{equation}\label{semigroup}
  T^{b,\eta}_t f = \int_{\T^d} q^{b,\eta}(t;\cdot,y)f(y) dy, \quad f\in \C(\T^d).
\end{equation}
By Proposition \ref{q-properties}, $\{T^{b,\eta}_t\}_{t\ge0}$ forms a (one-parameter operator) semigroup which is Markovian (positivity preserving, conservative and sub-Markovian) and Feller (each $T^{b,\eta}_t$ map $\C(\T^d)$ to $\C(\T^d)$). We can also prove the strong continuity. Hence, we have

\begin{proposition}\label{Feller}
  The family of operators $\{T^{b,\eta}_t\}_{t\ge0}$ forms a Feller semigroup on $\C(\T^d)$. Let $(\hat\L^{b,\eta},D(\hat\L^{b,\eta}))$ be the generator, then for all $\gamma>\alpha-1$, $\C^{1+\gamma}(\T^d)\subset D(\hat\L^{b,\eta})$ and $\hat\L^{b,\eta}=\L^{b,\eta}$ on $\C^{1+\gamma}(\T^d)$. Moreover, $\C^\infty(\T^d)$ is a core of $\hat\L^{b,\eta}$.
\end{proposition}

\begin{proof}
  (i). Fix $f\in\C(\T^d)$. For every $\e>0$, there is a constant $\delta>0$ such that $|f(x)-f(y)|<\e$ with $|x-y|<\delta$, $x,y\in\T^d$. Then by Proposition \ref{q-properties}.(i) and (iii),
  \begin{equation*}
    \begin{split}
       \sup_x \left| T^{b,\eta}_t f(x) - f(x) \right| &\le \sup_x \int_{\T^d} q^{b,\eta}(t;x,y)|f(y)-f(x)| dy \\
         & \le \e \sup_x \int_{\begin{subarray}{c} |x-y|<\delta \\ y\in\T^d \end{subarray}} q^{b,\eta}(t;x,y) dy + 2\|f\|_\infty \sup_x \int_{\begin{subarray}{c} |x-y|\ge\delta \\ y\in\T^d \end{subarray}} \varrho_\alpha(t;x-y+l) dy \\
         & \le \e + 2\|f\|_\infty t \int_{ |z|\ge\delta} \left( t^{-(d+\alpha)/\alpha } \wedge |z|^{-(d+\alpha)} \right) dz.
    \end{split}
  \end{equation*}
  When $t\to0^+$,
  $$\int_{ |z|\ge\delta} \left( t^{-(d+\alpha)/\alpha } \wedge |z|^{-(d+\alpha)} \right) dz \le \int_{ |z|\ge\delta} |z|^{-(d+\alpha)} dz < \infty,$$
  and then $\|T^{b,\eta}_t f-f\|_\infty\to0$. This proves that $\{T^{b,\eta}_t\}_{t\ge0}$ is strongly continuous on $\C(\T^d)$. Thus, $\{T^{b,\eta}_t\}_{t\ge0}$ is a Feller semigroup.

  (ii). To identify the generator of $\{T^{b,\eta}_t\}_{t\ge0}$, we fix $f\in\C^{1+\gamma}(\T^d)$ with $1+\gamma>\alpha$. We claim that for every $g\in\C^\infty(\T^d)$,
  \begin{equation}\label{generator}
    \lim_{t\to0} \int_{\T^d} \frac{1}{t}\left( T_t^{b,\eta}f(x)-f(x) \right) g(x) dx = \int_{\T^d} \L^{b,\eta}f(x)g(x) dx.
  \end{equation}
  Then using \cite[Theorem 1.24]{Dav80} and the fact that $\C^\infty(\T^d)$ is vaguely (i.e., weak-$*$) dense in the space $\M_b(\T^d)$ of all bounded signed Radon measures on $\T^d$, which is the topological dual of $\C(\T^d)$, we get that $\C^{1+\gamma}(\T^d)$ is contained in the domain of $\hat\L^{b,\eta}$, and the restriction of $\hat\L^{b,\eta}$ on $\C^{1+\gamma}(\T^d)$ equals to $\L^{b,\eta}$.

  Now we prove the claim \eqref{generator}. By \eqref{T-eta}, \eqref{semigroup} and \eqref{integral-eqn} we have
  \begin{equation*}
    \begin{split}
       &\ \int_{\T^d} \frac{1}{t}\left( T_t^{b,\eta}f(x)-f(x) \right) g(x) dx - \int_{\T^d} \L^{b,\eta}f(x)g(x) dx \\
       =&\ \int_{\T^d} \left[ \frac{1}{t}\left( T_t^{\eta}f(x)-f(x) \right) - \L^{\eta}f(x) \right] g(x) dx \\
       &\ + \frac{1}{t}\int_{\T^d}\left( \int_{\T^d} \int_0^t \int_{\T^d} q^{b,\eta}(t-s;x,z)b(z)\cdot\nabla_z q^\eta(s;z,y)f(y) dz ds dy - b(x)\cdot\nabla f(x)\right) g(x) dx \\
       =:&\ I + II.
    \end{split}
  \end{equation*}
  The term $I$ goes to zero, by Proposition \ref{q^eta}.(ii), as $t\to0$. For the term $II$, we use Fubini's theorem and integration by parts which we can do by the periodicity of $b,f,g$, $x\to q^{\eta}(t;x,y)$ and $x\to q^{b,\eta}(t;x,y)$, then we get
  \begin{equation*}
    \begin{split}
       II =&\ \frac{1}{t}  \int_0^t \int_{\T^d} \int_{\T^d} q^{b,\eta}(t-s;x,z) g(x) \left[ b(z)\cdot\left( \int_{\T^d} \nabla_z q^\eta(s;z,y) f(y)dy\right) - b(x)\cdot\nabla f(x) \right] dx dz ds \\
         =&\ \frac{1}{t}\int_0^t \int_{\T^d} \int_{\T^d} q^{b,\eta}(t-s;x,z) g(x) \int_{\T^d} \left[ b(z)\cdot \nabla_z q^\eta(s;z,y)-b(x)\cdot \nabla_x q^\eta(s;x,y) \right] f(y) dy dx dzds \\
         &\ + \frac{1}{t}\int_0^t \int_{\T^d} \int_{\T^d} q^{b,\eta}(t-s;x,z)g(x)b(x)\cdot\nabla_x\left[ \int_{\T^d} q^\eta(s;x,y)f(y)dy - f(x) \right] dxdzds \\
         =:&\ II_1 + \frac{1}{t}\int_0^t \int_{\T^d} \int_{\T^d} \nabla_x\left(q^{b,\eta}(t-s;x,z)g(x)b(x) \right) \cdot\left[ \int_{\T^d} q^\eta(s;x,y)f(y)dy - f(x) \right] dxdzds \\
         =:&\ II_1 + II_2.
    \end{split}
  \end{equation*}
  Since the function $(s,x,y)\to b(x)\cdot \nabla_x q^\eta(s;x,y)$ is uniformly continuous on $[0,t]\times\T^d\times\T^d$, there exists a constant $C>0$, such that $|b(x)\cdot \nabla_x q^\eta(s;x,y)|<C$ for all $(s,x,y)\in [0,t]\times\T^d\times\T^d$; and for every $\e>0$, there is $\delta>0$ such that $|b(z)\cdot \nabla_x q^\eta(s;z,y)-b(x)\cdot \nabla_x q^\eta(s;x,y)|<\e$ for $|x-z|<\delta$. Then by Proposition \ref{q-properties}.(iii), for $t\to0$,
  \begin{equation*}
    \begin{split}
       |II_1| \le&\ \|f\|_\infty\|g\|_\infty \bigg( \e \frac{1}{t}\int_0^t \iint_{\begin{subarray}{c} |x-z|<\delta \\ x,z\in\T^d \end{subarray}} q^{b,\eta}(t-s;x,z) dx dzds + 2c \frac{1}{t} \int_0^t \iint_{\begin{subarray}{c} |x-z|\ge\delta \\ x,z\in\T^d \end{subarray}} q^{b,\eta}(t-s;x,z) dx dzds \bigg) \\
         \le&\ \|f\|_\infty\|g\|_\infty \left( \e+2C \frac{1}{t} \int_0^t \int_{\T^d} \int_{|y|\ge\delta}\rho_\alpha(t;y) dy dz ds \right) \\
         \le&\ \|f\|_\infty\|g\|_\infty \left( \e+ 2Ct \int_{ |y|\ge\delta} |y|^{-(d+\alpha)} dy \right) \\
         \to&\ \e\|f\|_\infty\|g\|_\infty.
    \end{split}
  \end{equation*}
  Since $\e>0$ is arbitrary, $II_1\to0$ as $t\to0$. Moreover, the strong continuity of the semigroup $\{T_t\}_{t\ge0}$ and dominated convergence imply that $II_2\to0$ as $t\to0$. Thus, we get \eqref{generator}. For more general results of domains and representations of generators of Feller processes on $\R^d$, we refer the readers to \cite{KS19} and references therein.

  (iii). Finally, we prove that $\C^\infty(\T^d)$ is a core of the generator. We divide this proof into three steps.

  \emph{Step 1:} We prove that for every $f\in\C(\T^d)$ and all $t>0$, $T^{b,\eta}_t f$ is differentiable and the integral in $\L^{b,\eta}T^{b,\eta}_t f\in\C(\T^d)$ is absolutely integrable, and for all $x\in\T^d$,
  \begin{align}
    \nabla T^{b,\eta}_t f(x) & = \int_{\T^d} \nabla_x q^{b,\eta}(t;x,y) f(y) dy, \label{Tf-derivative} \\
    \L^{b,\eta} T^{b,\eta}_t f(x) & = \int_{\T^d} \L^{b,\eta}_x q^{b,\eta}(t;x,y) f(y) dy. \label{Tf-frac-derivative}
  \end{align}
  Using the estimate \eqref{q-b&eta-derivative} and writting the derivative as the limit of difference quotient, then \eqref{Tf-derivative} follows from the dominated convergence. Further, \eqref{Tf-frac-derivative} follows from \eqref{Tf-derivative} and Fubini's theorem. The continuity of the function $\L^{b,\eta} T^{b,\eta}_t f$ follows from the joint continuity of $\L^{b,\eta}_x q^{b,\eta}(t;x,y)$ and \eqref{Tf-frac-derivative}.

  \emph{Step 2:} Since the semigroup $\{T^{b,\eta}_t\}$ is Feller, its generator $(\hat\L^{b,\eta},D(\hat\L^{b,\eta}))$ is closed in $\C(\T^d)$ (see \cite[Definition 1.24]{BSW13}). By (i), we see that $(\L^{b,\eta},\C^\infty(\T^d)) \subset (\hat\L^{b,\eta},D(\hat\L^{b,\eta}))$, whence the former is closable in $\C(\T^d)$. Denote $(\bar\L^{b,\eta},\bar D):= \overline{(\L^{b,\eta},\C^\infty(\T^d))}$. In this step, we show that for every $f\in\C(\T^d)$ and all $t>0$, $T^{b,\eta}_tf \in \bar D$ and $\bar\L^{b,\eta}T^{b,\eta}_t f=\L^{b,\eta}T^{b,\eta}_t f$. Let $\{\phi_n\}_{n\in\N}$ be a standard mollifier such that $\mathrm{supp}(\phi_n)\subset B(0,1/n)$. Then $T^{b,\eta}_tf* \phi_n\in\C^\infty(\T^d)$ and $\|T^{b,\eta}_tf* \phi_n-T^{b,\eta}_tf\|_\infty\to 0$ as $n\to\infty$. By the definition of the closure $(\bar\L^{b,\eta},\bar D)$, it suffices to show that $\|\L^{b,\eta} (T^{b,\eta}_tf* \phi_n) - \L^{b,\eta}T^{b,\eta}_tf\|_\infty\to0$ as $n\to\infty$. Using \eqref{Tf-derivative}, \eqref{Tf-frac-derivative}, \eqref{q-b&eta-derivative}, \eqref{q-b&eta-frac-derivative}, \cite[Lemma 5.17.(a)]{GS19} and Fubini's theorem, we have
  \begin{equation*}
    \begin{split}
       &\ \left| \L^{b,\eta} (T^{b,\eta}_tf* \phi_n)(x) - (\L^{b,\eta}T^{b,\eta}_tf)* \phi_n(x) \right| \\
       \le&\ \left| \int_{\R^d}(b(x)-b(x-y))\cdot\nabla T^{b,\eta}_tf(x-y) \phi_n(y) dy \right| \\
       &\ +\bigg| \int_{\R^d}\int_{\Ro d}\left[ T^{b,\eta}_tf(x-y+z) - T^{b,\eta}_tf(x-y) - z\cdot \nabla T^{b,\eta}_tf(x-y) \ind_B(z) \right] \\
       & \qquad\qquad\qquad\quad\ \times \left(\kappa^\sharp(x,z) - \kappa^\sharp(x-y,z) \right) J(z)\phi_n(y) dydz \bigg| \\
       \le&\ \frac{1}{n^\beta} \|b\|_{\C^\beta} \| \nabla T^{b,\eta}_tf\|_\infty + \frac{1}{n^\beta} \frac{\kappa_3}{\kappa_1} \|\L^{b,\eta} T^{b,\eta}_t f\|_\infty \\
       \le&\ \frac{1}{n^\beta} C(c_0,T,b_0)\|f\|_\infty \left( \|b\|_{\C^\beta} t^{-\frac{1}{\alpha}} + \frac{\kappa_3}{\kappa_1} t^{-1}\right).
    \end{split}
  \end{equation*}
  Let $n\to\infty$, we get $\|\L^{b,\eta} (T^{b,\eta}_tf* \phi_n) - (\L^{b,\eta}T^{b,\eta}_tf)* \phi_n\|_\infty\to0$. Since $\L^{b,\eta}T^{b,\eta}_tf\in\C(\T^d)$ by Step 1, $(\L^{b,\eta}T^{b,\eta}_tf)* \phi_n \to \L^{b,\eta}T^{b,\eta}_tf$ in $\C(\T^d)$ as $n\to\infty$. Thus, we have $\|\L^{b,\eta} (T^{b,\eta}_tf* \phi_n) - \L^{b,\eta}T^{b,\eta}_tf\|_\infty\to0$ as $n\to\infty$, which ends this step.

  \emph{Step 3:} By construction, we have $(\bar\L^{b,\eta},\bar D) \subset (\hat\L^{b,\eta},D(\hat\L^{b,\eta}))$. 
  Now we show the converse, that is, for an arbitrary $f\in D(\hat\L^{b,\eta})$, we show that $f\in \bar D$ and $\bar\L^{b,\eta}f = \hat\L^{b,\eta}f$. Let $f_n=T^{b,\eta}_{1/n}f$. Since $\|f_n-f\|_\infty \to 0$ as $n\to\infty$, by the definition of the closure $(\bar\L^{b,\eta},\bar D)$, we only need to show $\|\bar\L^{b,\eta}f_n-\hat\L^{b,\eta}f\|_\infty \to 0$. From Step 2, we have $f_n\in \bar D$ and $\bar\L^{b,\eta}f_n = \hat\L^{b,\eta}f_n$. It follows that
  \begin{equation*}
    \|\bar\L^{b,\eta}f_n-\hat\L^{b,\eta}f\|_\infty = \|\hat\L^{b,\eta}f_n-\hat\L^{b,\eta}f\|_\infty = \|T^{b,\eta}_{1/n}\hat\L^{b,\eta}f-\hat\L^{b,\eta}f\|_\infty \to 0, \quad n\to \infty.
  \end{equation*}
  This gives $(\hat\L^{b,\eta},D(\hat\L^{b,\eta})) \subset (\bar\L^{b,\eta},\bar D)$ and thus $\bar\L^{b,\eta}=\hat\L^{b,\eta}$. We complete the whole proof.
\end{proof}

\section{SDEs and nonlocal PDEs}

The following result is a consequence of the nature of Feller semigroups (see \cite[Theorem 2.3, Corollary 2.5]{Kur11} and \cite[Theorem 4.4.1, Proposition 4.1.7]{EK09}).
\begin{corollary}\label{SDE-weak}
  The canonical Feller process $(X^{b,\eta};(\Omega,\F,\P))$ corresponding to $\{T^{b,\eta}_t\}_{t\ge0}$ with c\`adl\`ag trajectories is the unique solution to the martingale problem for $(\L^{b,\eta},\P\circ (X^{b,\eta}_0)^{-1})$, and also the unique weak solution to the following SDE
  \begin{equation}\label{SDE-no-diffusion}
    dX_t = b(X_t)dt + \int_0^\infty\int_{\Bo} \ind_{[0,\kappa^\sharp(X_{t-},z)]}(r) z \tilde N(dz,dr,dt) + \int_0^\infty\int_{B^c} \ind_{[0,\kappa^\sharp(X_{t-},z)]}(r) z N(dz,dr,dt),
  \end{equation}
  where $N$ is a Poisson random measure on $\R^d\times[0,\infty)\times[0,\infty)$ with intensity measure $J(z)dz\times m\times m$ and $\tilde N$ is the associated compensated Poisson random measure.
\end{corollary}

We have a generalized version of It\^o's formula as following. The proof is similar with that of \cite[Lemma 3.4]{Xie17} and shall be omitted.

\begin{lemma}\label{Ito}
  Let $f\in\C^{1+\gamma}(\T^d)$ with $1+\gamma>\alpha$. If $X$ satisfies the SDE \eqref{SDE-no-diffusion}, then
  \begin{equation*}
    \begin{split}
      f(X_t) - f(X_0) = &\ \int_0^t \L^{b,\eta}f(X_s) ds \\
         &\ + \int_0^t \int_0^\infty \int_{\Ro d} \left[ f(X_{s-}+\ind_{[0,\kappa^\sharp(X_{s-},z)]}(r)z) - f(X_{s-}) \right] \tilde N(dz,dr,ds).
    \end{split}
  \end{equation*}
\end{lemma}

We can solve the nonlocal Poisson with zeroth-order term, using the semigroup representation.

\begin{corollary}\label{Poisson-zero}
  For every $f\in\C^\beta(\T^d)$ and $\lambda>0$, there exists a unique classical solution $u\in\C^{\alpha+\beta}(\T^d)$ to the Poisson equation
  \begin{equation}\label{Poisson-0}
    \lambda u - \L^{b,\eta} u = f.
  \end{equation}
\end{corollary}

\begin{proof}
  We first prove that if $u_\lambda\in\C^{\alpha+\beta}(\T^d)$ is a solution of \eqref{Poisson-0}, then $u_\lambda$ must have the following representation
  \begin{equation}\label{Poisson-sol}
    u_\lambda(x) = \int_0^\infty e^{-\lambda t} T^{b,\eta}_t f(x) dt,
  \end{equation}
  and there exists a constant $C=C(c_0,b_0,\lambda)>0$ not depending on $f$ such that
  \begin{equation}\label{energy}
    \|u_\lambda\|_{\C^{\alpha+\beta}} \le C \|f\|_{\C^\beta}.
  \end{equation}
  Since the restriction of the generator $\hat\L^{b,\eta}$ on $\C^{\alpha+\beta}(\T^d)$ is $\L^{b,\eta}$, we have
  \begin{equation*}
    \begin{split}
      \int_0^\infty e^{-\lambda t} T^{b,\eta}_t f dt & = \int_0^\infty e^{-\lambda t} T^{b,\eta}_t (\lambda u_\lambda - \L^{b,\eta} u_\lambda) dt = -\int_0^\infty \frac{d}{dt}\left( e^{-\lambda t} T^{b,\eta}_t u_\lambda \right) dt \\
         & = u_\lambda - \lim_{t\to\infty} e^{-\lambda t} T^{b,\eta}_t u_\lambda = u_\lambda,
    \end{split}
  \end{equation*}
  where we have used the fact that $\|e^{-\lambda t} T^{b,\eta}_t u_\lambda\|_\infty\le e^{-\lambda t} \|u_\lambda\|_\infty\to 0$ as $t\to\infty$. This gives \eqref{Poisson-sol} and the uniqueness follows. Further, using the Schauder-type estimates in \cite[Theorem 7.1, Theorem 7.2]{Bas09}, there exist a constant $C=C(c_0,b_0,\lambda)>0$ such that
  \begin{equation*}
    \|u_\lambda\|_{\C^{\alpha+\beta}} \le C (\|u_\lambda\|_\infty+\|f\|_{\C^\beta}).
  \end{equation*}
  The representation \eqref{Poisson-sol} yields that
  \begin{equation*}
    \|u_\lambda\|_\infty \le \|f\|_\infty \int_0^\infty e^{-\lambda t}  dt = \frac{1}{\lambda} \|f\|_\infty.
  \end{equation*}
  The estimate \eqref{energy} follows.

  Moreover, it is shown in \cite[Theorem 3.4]{Pri12} that when the function $\kappa^\sharp$ is a constant, the existence and uniqueness hold in $\C^{\alpha+\beta}(\T^d)$. We can now obtain the existence of \eqref{Poisson-0} by the energy estimate \eqref{energy} and the method of continuity, see \cite[Section 5.2]{GT01}), also cf. \cite[Theorem 3.2]{HDS22}.
\end{proof}

\end{appendices}

\paragraph{Acknowledgements.}
The research of J. Duan was partly supported by the NSF grant 1620449. The research of Q. Huang was partly supported by China Scholarship Council (CSC), and NSFC grants 11531006 and 11771449. The research of R. Song is supported in part by a grant from the Simons Foundation ($\#$ 429343, Renming Song). We would like to thank Dr.~Yanjie Zhang for useful discussions.

{\footnotesize
\bibliographystyle{MyStyle-plainnat}
\bibliography{HomoFeller-ref}
}
\end{document}